\documentclass{amsart}
\usepackage{hyperref}
\usepackage{mathrsfs}
\usepackage{graphics}
\usepackage{color}
\usepackage{mathtools}  
\usepackage{tikz}
\usetikzlibrary{matrix,arrows,decorations}
\usepackage{tikz-cd}
\usepackage{amsmath}
\usepackage{amsthm}
\usepackage{amssymb}
\usepackage[normalem]{ulem}
\DeclareMathAlphabet{\mathcal}{OMS}{cmsy}{m}{n} 
\usepackage{hyperref}

\newtheorem{thm}{Theorem}[section]

\newtheorem{coro}[thm]{Corollary}
\newtheorem{prop}[thm]{Proposition}
\theoremstyle{definition}

\newtheorem{DMex}{Running example: De Morgan algebras}[]
\newtheorem*{DSex}{Double Stone algebras}
\newtheorem*{InvSex}{Involutive Stone algebras}

\newtheorem*{MSex}{MS-algebras}
\newtheorem*{Kiex}{Subvarieties of MS-algebras}  

\newcommand{\twiddle}[1]{\smash{\underset{\raise.375ex\hbox{$\smash\sim$}}
      {#1}}\vphantom{\underline{#1}}} 

\newcommand{\class}[1]{\mathcal{#1}}
\newcommand{\cat}[1]{\boldsymbol{\mathscr{#1}}}  
\newcommand{\str}[1]{\mathbf{#1}}
\newcommand{\alg}[1]{\mathbf{#1}}
\newcommand{\fnt}[1]{\mathsf{#1}}
\newcommand{\E}{\fnt{E}}
\newcommand{\D}{\fnt{D}}
\newcommand{\ope}[1]{\mathbb{#1}}
\newcommand{\defn}[1]{\emph{#1}}
 \newcommand{\showarticletitle}[1]{#1}

\newcommand{\CX}{\cat{X}}

\newcommand{\A}{\alg{A}}
\newcommand{\B}{\alg{B}}
\newcommand{\F}{\alg{F}}  
\newcommand{\M}{\alg{M}}
\newcommand{\X}{\str{X}}
\newcommand{\Y}{\str{Y}}
\newcommand{\Z}{\str{Z}}
\newcommand{\K}{\class{K}}
\newcommand{\Tp}{\mathcal{T}}
\newcommand{\MT}{\twiddle{\M}}
\newcommand{\CA}{\cat{A}}
\newcommand{\Q}{\class{Q}}

\newcommand{\medsub}[2]{#1\lower0.6ex\hbox{$\scriptstyle{#2}$}}
\newcommand{\epsub}[1]{\medsub \varepsilon {\kern-1.25pt #1}}
\newcommand{\esubA}{\medsub e {\kern-0.75pt\A\kern-0.75pt}}

\DeclareMathOperator{\ISP}{\ope{ISP}}
\DeclareMathOperator{\IScP}{\ope{IS}_c\ope{P}^{+}}
\DeclareMathOperator{\HSP}{\ope{HSP}}
 \DeclareMathOperator{\Con}{Con}
\DeclareMathOperator{\id}{id}  
\DeclareMathOperator{\im}{im}
\DeclareMathOperator{\Cong}{{\rm Con}}
\DeclareMathOperator{\End}{End}  
\newcommand{\Qg}{{\ensuremath{\mathbb{Q}}}}
\newcommand{\Vg}{{\ensuremath{\mathbb{V}}}}

\newcommand{\f}{\ensuremath{\varphi}}
\newcommand{\p}{\ensuremath{\psi}}
\newcommand{\si}{\sigma}
\newcommand{\imp}{\Rightarrow}
\newcommand{\mdl}[1]{\models_{#1}}

\renewcommand{\epsilon}{\varepsilon}  

\newcommand{\tafa}{\emph{TAFA}}

\newcommand{\xdoubleheadleftarrow}[1]{%
\leftarrow\mathrel{\mspace{-15mu}}{\xleftarrow{#1}}
}
\newcommand{\proc}[1]{{\textsc{#1}}}

\newcommand{\eq}{\approx}
\renewcommand{\leq}{\leqslant}
\renewcommand{\geq}{\geqslant}
\newcommand{\lang}{\mathcal{L}}

\begin{document}
\title[Checking Admissibility Using Natural Dualities]{Checking Admissibility Using Natural Dualities}  

\author{L.M. Cabrer}
\address{CNRS, IRIF, Univ. Paris-Diderot}
\author{B. Freisberg}
\address{Mathematical Institute, University of Bern}
\author{G. Metcalfe}
\address{Mathematical Institute, University of Bern}

\author{H. A. Priestley}
\address{Mathematical Institute, University of Oxford}

\begin{abstract}
This paper presents a new method for obtaining small algebras to check the admissibility---equivalently, 
validity in free algebras---of quasi-identities in a finitely generated quasivariety. Unlike a previous algebraic approach of Metcalfe and  R{\"o}thlisberger that is feasible only when the relevant free algebra is not too large, this method exploits natural dualities for quasivarieties to work with structures of smaller cardinality and surjective rather than injective morphisms. A number of case studies are described here that could not be be solved using the algebraic approach, including (quasi)varieties of MS-algebras, double Stone algebras, and involutive Stone algebras. 
\end{abstract}
%
%
%


\thanks{The first author of this paper received support from the European Research Council (ERC) under the European Union\'s Horizon 2020 research and innovation programme (grant agreement No.670624). The third author acknowledges support from Swiss National Science Foundation grant 200021{\_}146748.}

\maketitle

\renewcommand{\shortauthors}{L. Cabrer et al.}


\section{Introduction}

The concept of an admissible rule was introduced explicitly by Lorenzen in the 1950s in the context of intuitionistic propositional logic~\cite{Lor55}, but appeared already, at least implicitly, in the 1930s in papers by Gentzen~\cite{Gen35} and Johansson~\cite{Joh36}.  Informally, a rule, consisting of a finite set of premises and a conclusion, is derivable in a logical system if the conclusion can be obtained from the premises using the rules of the system, and admissible if adding it to the system produces no new theorems. In classical propositional logic, admissibility and derivability coincide---the logic is said to be structurally complete---but for many non-classical logics there exist admissible rules that are not derivable 
 (see, e.g.,~\cite{Ryb97,Ghi99,Ghi00,Iem01,Jer05,CM09,Jer09a,CM10,Met16}). Such admissible but non-derivable rules may be understood as  ``hidden properties'' of logical systems that can be used, for example, to establish completeness with respect to a certain class of algebras or to shorten derivations in the system.

From an algebraic perspective---where logics correspond to quasivarieties, and rules to quasi-identities---the admissible rules of a quasivariety $\Q$ are the valid quasi-identities of the free algebra of $\Q$ on a countably infinite set of generators.  If $\Q$ is locally finite, then a quasi-identity containing $k$ variables is admissible if and only if it is valid in the $k$-generated free algebra  $\F_\Q(k)$. Hence, if finitely generated free algebras are computable (equivalently, the equational theory of $\Q$ is decidable), checking admissibility in $\Q$ is decidable. Moreover, if $\Q$ is generated by a given class of $n$-generated algebras for some $n \in \mathbb{N}$,  then the $n$-generated free algebra  $\F_\Q(n)$ is computable and suffices for checking admissibility for any quasi-identity.  Nevertheless,  $\F_\Q(n)$ may be too large for checking validity of quasi-identities to be computationally feasible. For example, the (quasi)variety of De Morgan algebras is generated by a $2$-generated $4$-element algebra, but the free De Morgan algebra on $2$ generators has $168$ elements (see Running Example~\ref{DMex0} below). Naively using this algebra to check the admissibility of a rule with $n$ variables would involve considering $168^n$ possible evaluations.
 
This issue of feasibility is addressed by Metcalfe and R{\"o}thlisberger in~\cite{MR13}. 
These authors provide algorithms that for a finite set $\K$ of $n$-generated algebras
 generating a quasivariety $\Q$,  produce a finite set of ``small'' algebras 
 that admits the same valid quasi-identities as $\F_\Q(n)$, that is, 
 the admissible quasi-identities of $\Q$.\footnote{These algorithms have been implemented in a system called {\tafa} (a Tool for Admissibility in Finite Algebras) available to download from \url{https://sites.google.com/site/admissibility/downloads}.} 
 A first algorithm searches for (small) subalgebras of $\F_\Q(n)$ that have members of $\K$
as homomorphic images and therefore generate the same quasivariety. 
A second algorithm provides a generating set of algebras for a finitely generated 
quasivariety that is minimal with respect to the standard multiset well-ordering 
(see~\cite{DeMa79}) on the multiset of cardinalities of the algebras. 
The first algorithm is only feasible, however, when $\F_\Q(n)$ is of a manageable size, 
and the second is only feasible when the algebras in $\K$ are small. 
To get a rough idea what this means, note that these algorithms were used in~\cite{MR13} 
to obtain minimal sets of algebras for checking admissibility in quasivarieties 
such as De Morgan algebras where the relevant free algebra has fewer than $500$ elements, 
but are unable to handle cases such as double Stone algebras where the 
$2$-generated free algebra has $7\thinspace 776$ elements.

Natural dualities were proposed in~\cite{CM15} as a suitable framework for studying 
admissibility in finitely generated quasivarieties and used to obtain axiomatisations of 
admissible quasi-identities for several case studies. In this paper, we make further use of 
natural dualities to obtain new, more efficient, 
methods producing small algebras for checking admissibility in finitely generated quasivarieties. 
Suppose that a structure $\twiddle{\M}$ yields a strong duality on $\Q = \ope{ISP}(\M)$ 
where $\M$ is an $n$-generated algebra (see Section~\ref{sec:natdual}). 
Rather than construct the often prohibitively large free algebra $\F_\Q(n)$, 
we search for surjective morphisms from its dual $\twiddle{\M}^n$, which will have $|M|^n$ elements, 
onto a structure that contains the dual of $\M$.  
We present here a ``Test Spaces Method'' that combines this strategy with a dual version of the 
algorithm of~\cite{MR13} for obtaining minimal generating sets for finitely generated quasivarieties. 
We illustrate the method by using existing natural dualities for De Morgan algebras, MS-algebras, 
double Stone algebras, and involutive Stone algebras, to obtain small algebras for testing admissibility in these 
quasivarieties. Apart from De Morgan algebras, none of these case studies could be solved using the algebraic 
approach presented in~\cite{MR13}.


\section{Checking admissibility algebraically}\label{sec:admiss}

In this section, we provide the required background on quasivarieties, free algebras, 
and admissibility, and explain the algebraic methods developed in~\cite{MR13} 
for checking the admissibility of quasi-identities in finitely generated quasivarieties. 

For convenience, let us assume throughout this paper that $\lang$ is a finite algebraic language 
and that an \defn{$\lang$-algebra} is an algebraic structure for $\lang$. 
We denote the formula algebra (absolutely free algebra) for $\lang$ over
a countably infinite set of variables (free generators) by $\alg{Fm_\lang}$.  
An \defn{$\lang$-identity} is an ordered pair of $\lang$-formulas, written $\f \eq \p$. 
An \defn{$\lang$-quasi-identity} is an ordered pair consisting of a finite set of $\lang$-identities 
$\Sigma$ and an $\lang$-identity  $\f\eq\p$, written $\Sigma \imp \f\eq\p$.
As usual, we drop the prefix $\lang$ in referring to such notions when 
this is clear from the context.

Let $\K$ be a class of $\lang$-algebras. A quasi-identity $\Sigma \imp \f\eq\p$ 
is said to be \defn{valid} in $\K$, denoted by $\Sigma \mdl{\K}  \f\eq\p$,
 if it is satisfied by every $\alg{A} \in \K$: 
 that is, for each homomorphism $h \colon \alg{Fm_\lang} \to \alg{A}$, whenever 
 $h(\f') = h(\p')$ for all $\f' \eq \p' \in \Sigma$, also $h(\f) = h(\p)$. 
 If $\Sigma=\emptyset$, we  write simply $\mdl{\K} \f\eq\p$. 
 We call $\K$ a \defn{quasivariety} (\defn{variety}) if there exists a set $\Lambda$ of quasi-identities (identities) 
 such that  $\alg{A} \in \K$ if and only if $\alg{A}$ 
 satisfies all members of $\Lambda$. 
 The \defn{quasivariety $\Qg(\K)$} and \defn{variety $\Vg(\K)$ generated by} 
 $\K$ are, respectively, the smallest quasivariety and variety containing $\K$. 
 Let $\ope{H}$, $\ope{I}$, $\ope{S}$, $\ope{P}$, and $\ope{P}_U$ be the class operators 
 of taking homomorphic images, isomorphic images, subalgebras, products, and ultraproducts, 
 respectively. 
 Then $\Qg(\K) = \ope{ISPP}_U(\K)$ and $\Vg(\K) = \ope{HSP}(\K)$, 
 and if $\K$ is a finite set of finite algebras, $\Qg(\K) = \ope{ISP}(\K)$ 
 (see~\cite[Theorems~II.9.5 and V.2.25, and Lemma~IV.6.3]{BS81}).

For a cardinal $\kappa$, an $\lang$-algebra $\alg{B}$ is a 
\defn{free $\kappa$-generated algebra for $\K$}  
if there exists a set $X$ of cardinality $\kappa$ and a map $g\colon X\to \alg{B}$
such that $g[X]$ generates $\alg{B}$ and for every $\alg{A} \in \K$ and map 
$f \colon X \to \alg{A}$ there exists a (unique) homomorphism $h \colon \alg{B} \to \alg{A}$ 
satisfying $f=h \circ g$. 
In this case, each $x \in X$ is called a \defn{free generator} of $\alg{B}$, 
and the free algebra $\B$ is denoted by $\F_{\K}(\kappa)$. 
Note that $\F_{\K}(\kappa)$ might not belong to $\K$ but is always a member of $\Qg(\K)$. 
Note also that $\F_{\K}(\kappa)=\F_{\Qg(\K)}(\kappa)=\F_{\Vg(\K)}(\kappa)$, 
and we may therefore use $\F_{\K}(\kappa)$,  $\F_{\Qg(\K)}(\kappa)$, and 
$\F_{\Vg(\K)}(\kappa)$ interchangeably to denote the same algebra.

A quasi-identity $\Sigma \imp \f\eq\p$ is said to be \defn{admissible} in a class of 
$\lang$-algebras $\K$ if for every homomorphism (substitution) 
$\si \colon \alg{Fm_\lang} \to \alg{Fm_\lang}$,
\[
\mdl{\K} \si(\f') \eq \si(\p')\,  
\mbox{ for all }\f' \eq \p' \in \Sigma  \quad \Longrightarrow \quad \mdl{\K} \si(\f) \eq \si(\p).
\]
It is easily seen that the admissibility of a quasi-identity in $\K$ is equivalent to its validity in $\F_\K(\omega)$ 
(see, e.g.,~\cite{Ryb97}). Moreover, if $\K$ is a finite set of $n$-generated $\lang$-algebras for some 
$n \in \mathbb{N}$, then for any quasi-identity $\Sigma \imp \f \eq \p$,
\begin{equation}\label{Eq:AdmFree}
\tag{$\dagger$}
\Sigma \imp \f\eq\p \text{ is admissible in $\K$}  
\quad \Longleftrightarrow  \quad
\Sigma \mdl{\F_{\K}(n)} \f\eq\p.
\end{equation}
More generally, admissibility of quasi-identities in $\K$ is equivalent to validity in a class 
$\K'$ of $\lang$-algebras if and only if $\Qg(\K') = \Qg(\F_\K(\omega))$, 
which holds in turn if and only if $\K' \subseteq \Qg(\F_\K(\omega))$ and 
$\K \subseteq \Vg(\K')$  (see~\cite[Proposition~14]{MR13}). 

In~\cite{MR13}, Metcalfe and R{\"o}thlisberger introduced a method that, 
given any finite set $\K$ of $n$-generated $\lang$-algebras, produces a ``smallest'' set of
$\lang$-algebras $\K'$ such that admissibility of quasi-identities in $\K$ is equivalent 
to validity in $\K'$, that is, such that $\ISP(\K') = \ISP(\F_\K(n))$. 
To describe this method, and explain what ``smallest'' means in this context, 
let us first recall the standard multiset order and the notion of a minimal generating set 
for a finitely generated quasivariety. 
A \emph{finite multiset} over a set $S$ is an ordered pair $\langle S,f \rangle$,
where $f$ is a function $f\colon S \to \mathbb{N}$ and $\{ x \in S  \mid  f(x) > 0 \}$ is finite. 
As usual, we write  $[a_1,\ldots,a_n]$ to denote such a multiset where $a_1,\ldots,a_n \in S$ 
may include repetitions. 
If $\le$ is a well-ordering of $S$, then the \defn{multiset ordering $\leq_m$} on the set of all 
finite multisets over $S$ defined by
\[
\langle S,f \rangle \leq_m \langle S,g \rangle \ :\Leftrightarrow\ \forall x \in S \bigl( f(x) > g(x) \ \Longrightarrow \ \exists y \in S \bigl( y > x \text{ and } g(y) > f(y)\bigr) \bigr)
\]
is also a well-ordering~(see~\cite{DeMa79}). 

A set of finite $\lang$-algebras $\{ \A_1, \dots, \A_n \}$ is said to be a \defn{minimal generating set} for the quasivariety $\ISP(\A_1, \dots, \A_n)$
 if, for any set of finite $\lang$-algebras $\{\B_1, \dots, \B_k\}$,
\[
\ISP(\A_1, \dots, \A_n) = \ISP(\B_1, \dots, \B_k) 
\ \Longrightarrow \
[|A_1|, \dots, |A_n|] \leq_m [|B_1|, \dots, |B_k|].
\]
The choice of this order is based on the fact that if $[|A_1|, \dots, |A_n|] \leq_m [|B_1|, \dots, |B_k|]$, then for quasi-identities with sufficiently many variables, checking validity in  $\{ \A_1, \dots, \A_n \}$  involves considering fewer assignments of variables than checking validity in $\{ \B_1, \dots, \B_k \}$. Hence our goal will be to obtain a set of generators that is minimal in the multiset order. It is proved in~\cite{MR13} that for finitely generated quasivarieties such a set always exists and is unique up to isomorphism (see Theorem~\ref{t:MinGenSet}).

\section{The algorithms \proc{MinGenSet} and \proc{SubPreHom}}\label{algos}

To understand how the algorithm $\proc{MinGenSet}$ works, we first present a criterion 
for a set of finite algebras to be the minimal  generating set (up to isomorphism) for a finitely 
generated quasivariety~$\Q$. Let $\A$ be any $\lang$-algebra and let $\Con(\A)$ denote 
the congruence lattice of $\A$ with bottom element (i.e.,  identity relation) $\Delta_{\A}$. 
We call $\theta \in \Con(\A)$ a \defn{$\Q$-congruence} if $\A/\theta\in\Q$. The set of 
$\Q$-congruences of $\A$ is then a lattice $\Cong_{\Q}(\A)$ under set-inclusion, 
and a meet subsemilattice of $\Con(\A)$. Moreover, since $\Q$ is a quasivariety 
(in particular, closed under products and isomorphic images), there is a minimal congruence 
$\theta$ such that $\A/\theta\in\Q$, and, trivially, $\A \in \Q$ if and only if $\Delta_{\A}\in\Con_{\Q}(\A)$. 

An $\lang$-algebra $\A$ is said to be \defn{$\Q$-subdirectly irreducible} if whenever $\A$ is a subdirect product of algebras in  $\Q$, it is isomorphic to one of the components, or, equivalently, if $\Delta_{\A}$ is completely meet-irreducible in the lattice $\Cong_{\Q}(\A)$. If $\A$ is finite, then $\Con(\A)$ is also finite and the completely meet-irreducible and meet-irreducible elements of  $\Cong_{\Q}(\A)$ coincide. Moreover, if $\K$ is a finite set of $\lang$-algebras and $\Q=\ISP(\K)$,  then $\Q=\ISP(\{\A/\theta\mid \A\in\K \text{ and }\theta\text{ is meet-irreducible in }\!\Cong_{\Q}(\A)\})$ (see~\cite[Corollary 6]{XC80}). 

\begin{thm}[{\cite[Theorems 4 and 8]{MR13}}] \label{t:MinGenSet}
Let $\Q$ be a finitely generated quasivariety. Then there exists a finite set of finite $\Q$-subdirectly irreducible algebras $\{\A_1, \dots, \A_n \}$ such that $\Q = \ISP(\A_1, \dots, \A_n )$ and $\A_i \not \in \ope{IS}(\A_j)$ for each $j \in \{1,\ldots,n\} {\setminus} \{i\}$. Moreover, $\{ \A_1, \dots, \A_n \}$ is (up to isomorphism) the unique minimal generating set for~$\Q$.
\end{thm}

We now present the algorithm $\proc{MinGenSet}$ given in~\cite{MR13} for obtaining the minimal generating set of a finitely generated quasivariety~$\Q$. Take as input a finite set of finite $\lang$-algebras $\K$ with $\Q = \ISP(\K)$ and let $\mathcal{M}$ be a list containing the algebras in $\K$, setting $i=1$:

\begin{enumerate} 

\item[1.]	\texttt{For $\A$ at position $i$ in $\mathcal{M}$, determine the set $S_1$  (resp.~$S_2$) of congruences $\theta \in \Con(\A) {\setminus} \{\Delta_\A\}$ for which  $\A / \theta$ embeds into $\A$ (resp.~another member of $\mathcal{M}$).}

\item[2.]	\texttt{If $\bigcap (S_1 \cup S_2) = \Delta_\A$, then add  $\A / \theta$ to $\mathcal{M}$ for each $\theta \in S_1 {\setminus} S_2$ and remove $\A$, otherwise set $i$ to $i+1$.}

\item[3.]	\texttt{If $i \leq {\rm length}(\mathcal{M})$, then repeat from 1.}
\item[4.] \texttt{Remove from $\mathcal{M}$ any algebra that is a proper subalgebra of another member of~$\mathcal{M}$, and output $\mathcal{M}$ as a set.}
\end{enumerate}

\noindent
Note that Step 2 is key here. If $\bigcap (S_1 \cup S_2) = \Delta_\A$, then clearly $\A$ embeds into $\prod_{\theta\in S_1\cup S_2} \A/\theta$. Using the fact that $\K$ generates $\Q$, it follows that $\A$ is $\Q$-subdirectly irreducible if and only if $\bigcap (S_1 \cup S_2) \neq \Delta_\A$ (see~\cite[Lemma 7]{MR13}). If $\A$ is $\Q$-subdirectly irreducible, the algorithm proceeds to the next algebra in the list $\mathcal{M}$; otherwise it adds to $\mathcal{M}$ all the quotients $\A / \theta$ for $\theta \in \Con(\A) {\setminus} \{\Delta_\A\}$ that embed into $\A$, but not into any other member of  $\mathcal{M}$.
 
The algorithm $\proc{MinGenSet}$ involves calculating congruence lattices of 
finite algebras  and this takes exponential time.
It is therefore only feasible when the algebras $\A_i$ are small. 
Indeed, with its breadth-first approach, the algorithm 
calculates more lattices of congruences than are strictly necessary. 
Once $\Cong_{\Q}(\A)$ has been calculated for some $\A \in \K$, 
also the structure of the congruence lattices of its quotients is known.  
Hence we can recognise which congruences  $\theta\in\Cong_{\Q}(\A)$ are such that 
$\A/\theta$ is $\Q$-subdirectly irreducible. 
These are exactly the meet-irreducible elements of the lattice $\Cong_{\Q}(\A)$. 
Hence we can improve the algorithm by applying a depth-first approach. 
Let $\mathcal{S}$ initially be the empty set.  
Then the refined algorithm proceeds as follows:

\begin{enumerate} 
\item[1.]	\texttt{Given $\A\in \K$,  calculate $\Con_{\Q}(\A)$.}

\item[2.]	\texttt{Determine the set $L$ of (completely) meet-irreducible elements of  
$\Con_{\Q}(\A)$.}
\item[3.] \texttt{Add to $\mathcal{S}$ the algebras $\A/\theta$ such that $\theta$ 
is a minimal element of $L$.}
\item[4.] \texttt{Delete $\A$ from $\K$ and if $\K$ is non-empty, repeat from 1.}
\item[5.] \texttt{Remove from $\mathcal{S}$ any algebra that is a proper subalgebra of another 
member of~$\mathcal{S}$, and output $\mathcal{S}$.}
\end{enumerate}

The algorithm \proc{MinGenSet} described above applies to \emph{any} finite generating set of finite algebras for a quasivariety. Our specific interest in this paper lies, however, with finding small algebras that can be used to test admissibility. That is, we consider a quasivariety generated by a finite set $\K$ of $n$-generated $\lang$-algebras, and seek a minimal generating set for the quasivariety $\Q= \ISP(\F_\K(n))$.  The free algebra $\F_\K(n)$ is often very large (see for example Table~\ref{table:casestud} on page~\pageref{table:casestud}), and it is therefore useful to first seek smaller algebras that generate $\Q$, rather than attempting to apply \proc{MinGenSet} directly to $\F_\K(n)$, which would requires a description of $\Con_{\Q}(\F_{\K}(n))$. To this end, an algorithm $\proc{SubPreHom}$ is defined in~\cite{MR13} that searches for a smallest subalgebra of a finite algebra $\A$ such that another finite algebra $\B$ is a homomorphic image of $\A$. In particular, if $\K = \{\B_1,\dots,\B_m\}$, and $\alg{C}_1,\ldots,\alg{C}_m$ are subalgebras of $\F_\K(n)$ such that $\B_i$ is a homomorphic image of $\alg{C}_i$ for each $i$, then $\Q = \ISP(\alg{C}_1,\ldots,\alg{C}_m)$. Although this step is optional in the sense that $\proc{MinGenSet}$ can be applied directly to $\F_\K(n)$, applying $\proc{SubPreHom}$ already gives a 
best-possible result in certain cases.

We now consider in the present context the well-studied example in which $\K$ consists of a single $4$-element algebra generating the (quasi)variety of De Morgan algebras.


\begin{DMex}\label{DMex0}
A \defn{De Morgan algebra} $\A=(A; \wedge, \vee, \neg,0,1)$ consists of a  bounded distributive lattice $(A; \wedge, \vee, 0, 1)$ equipped with a unary operation $\neg$ satisfying $\neg\neg a=a$ and $\neg(a\wedge b)=\neg a \vee \neg b$ for all $a,b\in A$. Let us denote by $\class{DM}$ the variety of all De Morgan algebras and  by $\alg{D_4}$ the De Morgan algebra represented in Fig.~\ref{Fig:DeMorgan}(a) (using arrows to depict the action of~$\neg$). Then $\class{DM} = \HSP(\alg{D_4}) = \ISP(\alg{D_4})$  (see~\cite{Kal58} or~\cite[Chapter XI, Section~2]{BaDw}).

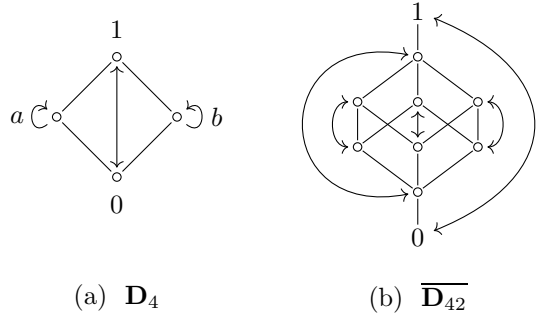
\begin{figure}[ht]
\begin{center}
	\begin{tikzpicture}[scale=.8]  
		\node [label=above:{$1$}](D1) at (0,3) {};   
		\node [label=below:{$0$}](D2) at (0,1) {}; 
		\node (D3) at (-1,2) {};  
		\node (D4) at (1,2) {}; 	
		\node [xshift=-15pt] at (D3) {$a$};
		\node [xshift=15pt] at (D4) {$b$};
		
\draw [<->]  (D1) to (D2);
\draw [<-]  (D3) .. controls +(-.5,.4) and +(-.5,-.4) .. (D3);
\draw [<-]  (D4) .. controls +(.5,.4) and +(.5,-.4) .. (D4);

		\draw [shorten <=-2pt, shorten >=-2pt] (D1) -- (D3);
		\draw [shorten <=-2pt, shorten >=-2pt] (D1) -- (D4);
		\draw [shorten <=-2pt, shorten >=-2pt] (D2) -- (D3);
		\draw [shorten <=-2pt, shorten >=-2pt] (D2) -- (D4);
		
		\draw (D1) circle [radius=2pt];
		\draw (D2) circle [radius=2pt];
		\draw (D3) circle [radius=2pt];
		\draw (D4) circle [radius=2pt];
		\node at (0,-1) {(a) \ $\alg{D}_{4}$};   

\node at (5,-1) {(b) \ $\overline{\alg{D}_{42}}$};   
		\node (1) at (5,0) {$0$}; 
		\node (2) at (5,3) { };   
		\node (3) at (5,2.25) {}; 
		\node (4) at (5,1.5) {}; 
		\node (5) at (5,.75) {}; 
		\node (6) at (4,2.25) {};  
		\node (7) at (4,1.5) {}; 
		\node (8) at (6,2.25) {}; 
		\node (9) at (6,1.5) {}; 
		\node (10) at (5,3.75) {$1$};

		\draw (2) circle (2pt);
		\draw  (3) circle (2pt);
		\draw  (4)  circle (2pt);
		\draw  (5) circle (2pt);
		\draw  (6) circle (2pt);
		\draw  (7) circle (2pt);
		\draw  (8) circle (2pt);
		\draw  (9) circle (2pt);

   		\draw [shorten <=-2pt, shorten >=-1pt] (1) -- (5);
		\draw [shorten <=-2pt, shorten >=-2pt] (5) -- (7);
		\draw [shorten <=-1pt, shorten >=-1pt] (5) -- (4);
		\draw [shorten <=-2pt, shorten >=-2pt] (5) -- (9);
		\draw [shorten <=-1pt, shorten >=-1pt] (7) -- (6);
		\draw [shorten <=-1pt, shorten >=-1pt] (9) -- (8);
		\draw [shorten <=-2pt, shorten >=-2pt] (2) -- (8);
		\draw [shorten <=-1pt, shorten >=-1pt] (2) -- (3);
		\draw [shorten <=-2pt, shorten >=-2pt] (2) -- (6);
		\draw [shorten <=-1pt, shorten >=-2pt] (2) -- (10);
		\draw [shorten <=-2pt, shorten >=-2pt] (7) -- (3);
		\draw [shorten <=-2pt, shorten >=-2pt] (9) -- (3);
		\draw [shorten <=-2pt, shorten >=-2pt] (6) -- (4);
		\draw [shorten <=-2pt, shorten >=-2pt] (8) -- (4);


\draw [<->]  (1) .. controls +(2.5,1) and +(2.5,-1) .. (10);
\draw [<->]  (2) .. controls +(-2.5,.5) and +(-2.5,-.5) .. (5);
\draw [<->]  (3) -- (4);
\draw [<->]  (6) .. controls +(-.5,0) and +(-.5,0) .. (7);
\draw [<->]  (8) .. controls +(.5,0) and +(.5,0) .. (9);
	\end{tikzpicture}
\end{center}
\caption{De Morgan algebras}\label{Fig:DeMorgan}
\end{figure}
\noindent

Since $\alg{D_4}$ is $2$-generated, the admissibility of a quasi-identity in $\class{DM}$ is 
equivalent to its validity in the free algebra $\F_{\class{DM}}(2)$. 
This algebra has cardinality  $168$ (the $4^\text{th}$ Dedekind number),
so using it to check the admissibility of a quasi-identity with, e.g.,~$3$  variables, would require 
considering $168^3=4\,741\,632$ possible evaluations. 
In~\cite{MR12}, Metcalfe and R{\"o}thlisberger proved, however, 
that the (much smaller) $10$-element De Morgan algebra $\overline{\alg{D}_{42}}$ 
(see Fig.~\ref{Fig:DeMorgan}(b)) generates the same quasivariety as $\F_{\class{DM}}(2)$. 
In~\cite{MR13}, the same authors used \proc{MinGenSet} and \proc{SubPreHom}
 to confirm that $\{\overline{\alg{D}_{42}}\}$ is indeed the minimal generating set for this quasivariety. 
\end{DMex}


\section{Natural dualities}\label{sec:natdual}

This section  recalls very briefly the theory of natural dualities, 
noting that a textbook treatment is given in~\cite{CD98}.  

Let $\M$ be a finite algebra. 
Depending on the context, we denote by $\ISP(\M)$ both the quasivariety generated by $\M$ and 
the category $\CA$ consisting of algebras from this quasivariety as objects 
with homomorphisms between algebras as the morphisms. 
Our aim is to find a second category~$\CX$, whose objects are 
topological structures of common type that is dually equivalent to~$\CA$ 
via functors $\D \colon \CA \to \CX$ and $\E \colon \CX  \to \CA$. 
We consider a topological structure $\MT = (M; G,H,R,\Tp)$ where 

\begin{enumerate}

\item[$\bullet$] $\Tp$ is the discrete topology on~$M$;

\item[$\bullet$] $G$ is a set of operations on $M$, meaning that, for $g \in G$ of arity $n \geq 0$, the map $g \colon \M^n \to \M$ is a homomorphism;

\item[$\bullet$] $H$ is a set of partial operations on $M$, meaning that, for $h \in H$ of arity $n \geq 1$, 
the map $h$ is a homomorphism from a (proper) subalgebra of $\M^n$ into $\M$;
 
\item[$\bullet$]  $R$ is a set of  relations on~$M$ such that if $r\in R$ is $n$-ary   ($n \geq 1$), then~$r$ is the universe of a subalgebra~$\alg{r}$ of~$\M^n$.

\end{enumerate} 

\noindent
We refer to such a topological structure $\MT$ as an \defn{alter ego} for~$\M$ and say that 
$\MT$ and $\M$ are \defn{compatible}. 
(We shall not encounter any alter egos with $H \ne \emptyset $ in the present paper, 
but permitting partial endomorphisms is crucial in certain applications of the 
Test Spaces Method; see~\cite{CM18}.) 

Using $\MT$ we build the desired category~$\CX$ of structured topological spaces. 
We first note that for any non-empty set~$S$ we may equip $M^S$ with 
the product topology and lift the members of $G \cup H \cup R$ pointwise to $M^S$. 
We define $\CX  := \IScP(\MT)$, the class of isomorphic copies of closed substructures of 
non-empty powers of $\MT$ together with the empty structure. 
Here a non-empty power $\MT^S$ of  $\MT$ carries the product topology and 
is equipped with the pointwise liftings of the members of $G\cup H \cup R$. 
Closed substructures and isomorphic copies are defined in the expected  way. 
Hence a member $\X$ of $\CX$ is a structure $(X; G^\X, H^\X, R^\X,\Tp^\X)$ of 
the same type as $\MT$. 
Details are given in~\cite[Section 1.4]{CD98}.  
We make~$\CX$ into  a category by taking all continuous structure-preserving maps 
as the morphisms. 
An \defn{embedding} in $\CX$ is a morphism $\phi \colon \X \to \Y$ such that $\phi(\X)$ 
is a substructure of~$\Y$ and $\phi\colon \X \to \phi(\X)$ is an isomorphism. 

Suppose now that a structured topological space $\MT$ and an algebra~$\M$ are compatible 
and let~$\A \in \CA$ and $\X \in\CX$.  
Then $\CA(\A,\M)$, the set of homomorphisms from $\A$ to $\B$, 
is the universe of a closed substructure of $\MT^{A}$, and $\CX(\X,\MT)$, 
the set of morphisms from $\X$ to $\MT$, is the universe of a subalgebra of $\M^{X}$. 
As a consequence of compatibility, there exist well-defined contravariant hom-functors 
$\D \colon \CA \to \CX$ and $\E \colon \CX  \to \CA$,
\smallskip
\begin{center}
\begin{tabular}{lcl}
\text{on objects:} &\phantom{.}\hspace{2.5cm}\phantom{.}& 
$\D \colon  \A \mapsto  \CA(\A,\M),$ 
\\[.1cm]
\text{on morphisms:}  & &  $\D \colon  x \mapsto - \circ x,$
\\[.1cm]
and 
\\[.1cm]
\text{on objects:} & & $\E  \colon  \X \mapsto  \CX (\X,\MT), $
\\[.1cm]
\text{on morphisms:}  & &$\E  \colon  \phi \mapsto - \circ \phi.$
\end{tabular}
\end{center}
\smallskip
For $\A \in \CA$ we refer to $\D(\A)$ as the \defn{{\rm (}natural{\rm )} dual space of} $\A$. 

Given $\A\in \CA$ and $\X \in \CX$,  there exist natural evaluation maps 
$\esubA \colon a \mapsto -\circ a$ and $\epsub{\X} \colon x \mapsto -\circ x$, 
with $\esubA \colon \A \to \E\D(\A)$ and $\epsub{\X}\colon \X \to \D\E(\X)$. 
Moreover,  $(\D,\E,e,\epsilon)$ is a dual adjunction (see~\cite[Chapter~2]{CD98}). 
Each of the maps $\esubA$ and $\epsub{\X} $ is an embedding.  
We say that~$\MT$ \defn{yields a duality on}~$\CA$ if each $\esubA$ is also surjective. 
If in addition each $\epsub{\X}$ is surjective and so an isomorphism, 
we say that the duality yielded by~$\MT$ is \defn{full}. 
In this case, $\CA$ and~$\CX$ are dually equivalent. 
Let us note already  here a fact which will be of key importance. 
A dualising alter ego~$\MT$ plays a special role in the duality it sets up: 
it is the dual space of the free algebra on one generator in~$\CA$. 
More generally, the free algebra generated by a non-empty set~$S$ has dual space~$\MT^S$.

The classes of algebras we consider in this paper are all lattice-based. 
In this setting, the \textit{existence} of some dualising alter ego (with $H = \emptyset$) 
is ensured by the NU Duality Theorem; see~\cite[Chapter~2,  Theorem 3.4]{CD98}. 
Moreover, $\M$ will be a finite distributive lattice with additional operations.  
Priestley duality for distributive lattices provides a prototypical natural duality 
in which $\M$ is the lattice with universe $\{ 0,1\}$ and $\MT$ is the discretely topologised 
poset with $0 < 1$; the members of the dual category are \defn{Priestley spaces}. 
(See~\cite{CD98} for details.) 
For many varieties of distributive lattice-based algebras, it is possible to find a 
dually equivalent category whose objects are certain Priestley spaces 
equipped with additional structure.  
Rather imprecisely, we refer to such equivalences as restricted Priestley dualities. 
We stress that these dual representations are seldom natural dualities. 
However, in a few  cases they can be recast as such; De Morgan, Stone, 
and double Stone algebras all have this special feature. 
In such cases we can take  advantage of basic facts about both types of duality 
to facilitate calculations. 
In particular, the cardinalities of free algebras are easy to compute. 
(Let us note also that restricted Priestley dualities can be useful tools even when 
they are not natural. Indeed they can be exploited to good effect in the construction of 
alter egos.)
 
\begin{DMex}  \label{DMex1}  
A topological duality for De Morgan algebras was originally developed in~\cite{CF77} 
in the guise of a restricted Priestley duality. 
The same result viewed from a natural duality perspective may be 
found in~\cite[Chapter 4, 3.15]{CD98}. 

Recall that $\class{DM}$ is generated as a quasivariety by the algebra $\alg{D_4}$ with 
universe $\{0,a,b,1\}$, as depicted in part~\ref{DMex0} of this running example. 
Let us consider now $\twiddle{\alg{D_4}}=(\{0,a,b,1\};\preccurlyeq,g)$ as shown in 
Fig.~\ref{fig:DMalterego}, where $\preccurlyeq$ is the partial order 
and~$g$ is the indicated order-reversing map. 
It is easy to check that $\preccurlyeq$ is the universe of a subalgebra of $\alg{D_4}^2$ 
and that $g\colon D_4\to D_4$ corresponds to the only non-identity 
endomorphism of~$\alg{D_4}$.  
Hence $\alg{D_4}$ and $\twiddle{\alg{D_4}}$ are compatible. 
Moreover, as proved in~\cite[Chapter 4, 3.15]{CD98}, $\twiddle{\alg{D_4}}$ yields a strong 
duality on $\class{DM}$.

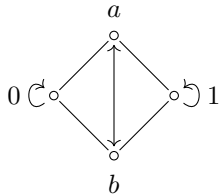
\begin{figure} [ht]
\begin{center}
	\begin{tikzpicture}[scale=.8]  	
		\node [label=above:{$a$}](D1) at (0,3) {};   
		\node [label=below:{$b$}](D2) at (0,1) {}; 
		\node  (D3) at (-1,2) {};  
		\node (D4) at (1,2) {}; 
		\node [xshift=-15pt] at (D3) {$0$};
		\node [xshift=15pt] at (D4) {$1$};
		
		\draw (D1) circle [radius=2pt];
		\draw (D2) circle [radius=2pt];
		\draw (D3) circle [radius=2pt];
		\draw (D4) circle [radius=2pt];

\draw [<->]  (D1) to (D2);
\draw [<-]  (D3) .. controls +(-.5,.4) and +(-.5,-.4) .. (D3);
\draw [<-]  (D4) .. controls +(.5,.4) and +(.5,-.4) .. (D4);

		\draw [shorten <=-2pt, shorten >=-2pt] (D1) -- (D3);
		\draw [shorten <=-2pt, shorten >=-2pt] (D1) -- (D4);
		\draw [shorten <=-2pt, shorten >=-2pt] (D2) -- (D3);
		\draw [shorten <=-2pt, shorten >=-2pt] (D2) -- (D4);
\end{tikzpicture}
\end{center}
\caption{The alter ego of $\alg{D_4}$ \label{fig:DMalterego}}
\end{figure}
\end{DMex}

We now begin to home in  on those aspects of duality theory that underpin this paper.
Our objective in setting up a natural duality for a quasivariety~$\CA$ is to thereby 
transfer algebraic problems about~$\CA$ into problems about the dual category~$\CX$ 
using the hom-functors~$\D$ and~$\E$ to toggle backwards and forwards.  
If we have a  duality,  we can identify any  $\A\in \CA$ with its second dual~$\E\D(\A)$.  
Making such identifications leads also to identifications of morphisms with their second duals.     
In addition---and this will be crucial for our Test Spaces Method---provided the duality is full, 
we  can realise any $\X \in \CX$, up to isomorphism, as $\D(\A) $ 
for some $\A\in\CA$.


Our needs  are very specific, relating  to the material in Section~\ref{sec:admiss}.  
We have already mentioned the dual spaces of free algebras.
We also require dual characterisations  of injective homomorphisms and 
surjective homomorphisms. 
This is not a triviality since, for  morphisms in~$\CX$,  epi (mono) may not equate 
to surjective (injective).  
As the discussion of $\proc{MinGenSet}$ in Section~\ref{sec:admiss} foretells, 
we also need to express notions and results concerning $\CA$-congruences in dual form. 

All the special properties we need will hold so long as our duality is \defn{strong}.   
Most concisely, a duality between $\CA = \ISP(\M)$ and $\CX = \IScP(\twiddle{\M})$ is \defn{strong} if 
$\twiddle{\M}$ is injective in~$\CX$. 
The technical details, and various equivalent definitions, need not concern us here 
(they can be found in~\cite[Chapter~3]{CD98}). 
We shall exploit without  proof two key facts. 
The first is that any strong duality is full.  
The second is that each of~$\D$ and $\E$ has the property that it converts 
embeddings to surjections and surjections to embeddings. 
In the case that~$\M$ is a finite lattice-based algebra, the existence of  an alter ego
which  yields a strong duality is guaranteed.  
Indeed, strongness  may be achieved by enriching a dualising alter ego by adding suitable
partial operations. 
However, for the quasivarieties we consider in this paper,
the known dualities we call on have no partial operations in their alter egos and  
are already strong.


\section{Dual formulation of \proc{MinGenSet}}\label{Sec:MGS}

Throughout the rest of the paper $\M$ denotes  a finite algebra and $\CA=\ISP(\M)$ 
the quasivariety it generates. 
Assume also that the structure $\twiddle{\M}$ yields a strong duality on $\ISP(\M)$ 
with associated contravariant functors 
$\D\colon \CA\to \IScP(\twiddle{\M})$ and $\E\colon\IScP(\twiddle{\M})\to \CA$. 
Because all the algebras we work with are finite, topology plays no overt role.

Our aim in this section is to dualise the procedure $\proc{MinGenSet}$ presented in Section~\ref{algos}. 
To achieve this, we need to capture in dual form the lattice of $\Q$-congruences
of an algebra, for some quasivariety $\Q\subseteq\CA$. As a first  step, we spell out in detail the dual characterisation of $\CA$-congruences 
\cite[Chapter 3, Theorem.2.1]{CD98} as this applies to finite algebras.  

\begin{prop}  \label{prop:cong-dualchar}
Let  $\A$ be a finite algebra in $\CA=\ISP(\M)$ and let ${\X} = \D(\A)$.  

\begin{enumerate}

\item[\rm (a)]
For any congruence $\theta$ of~$\A$ such that $\A/\theta \in \CA$,  there exists a substructure of~$\X$ given by $\Z_\theta = \im \D(f)$, where $f \in \CA(\A, \A/\theta)$ is the natural projection.

\item[\rm (b)]
For any substructure $\Z$ of\, $\X$, there exists an $\CA$-congruence $\theta_{\Z}$ of~$\A$  
 given by 
\[
 (a,b) \in \theta_{\Z} \ \Leftrightarrow \ z(a) = z(b)\text{ for all } z \in Z .
\] 
\end{enumerate}
The correspondence  set up by  $\theta \mapsto \Z_\theta$ and $\Z\mapsto \theta_{\Z}$ 
defines a dual order-isomorphism between the 
$\CA$-congruences of  $\A$ and the  substructures of\, $\X$, both ordered by inclusion.
\end{prop}

It follows from Proposition~\ref{prop:cong-dualchar} that the family of substructures of $\X = \D(\A)$  is itself a (complete) lattice with respect to the inclusion order.

Let us recall that to check admissibility in the quasivariety $\ISP(\M)$ we aim to find 
the minimal set of generators of the quasivariety $\Q=\ISP(\F_{\ISP(\M)}(s))$, 
where $s$ is the cardinality of a set of generators of $\M$. Algebraically, this involves 
applying \proc{MinGenSet} to some finite $\A\in \ISP(\M)$, which requires determining 
$\Con_{\Q}(\A)$. With this in mind, let us first investigate when an algebra in $\ISP(\M)$ belongs to 
the quasivariety generated by another algebra in $\ISP(\M)$.

\begin{prop}\label{lem:Dist}  
Let $\B$ and $\alg{C}$ be finite algebras in $\CA$.
 Then the following are equivalent:
\begin{enumerate}
\item[{\rm (1)}]  $\alg{C}\in \ISP(\B)$;

\item[{\rm (2)}]	there exist homomorphisms $f_1,\ldots, f_m\colon \alg{C} \to \B$ such that $f\colon a \mapsto (f_1(a), \ldots, f_m(a))$ is an injective homomorphism from $\alg{C}$ into $\B^m$;

\item[{\rm (3)}]  there exist  finitely many morphisms, $\phi_1, \ldots, \phi_m\in \CX(\D(\B),\D(\alg{C}))$ such that  
\[
\D(\alg{C})=\langle \im(\phi_1)\cup \cdots\cup\im(\phi_m)\rangle.
\]
\end{enumerate}
\end{prop}
\begin{proof} 
The equivalence of (1) and (2) may be seen as a specialisation to finite~$\alg{C}$ of ~\cite[Chapter~1, Theorem~3.1]{CD98}.
It is immediate that (2) implies (1).
Conversely,  if $\alg{C}\in \ISP(\B)$,  then the homomorphisms from~$\alg{C}$ into~$\B$ 
separate the points of~$C$: for $c \ne d$ in~$C$ there exists a homomorphism 
$g_{cd} \colon \alg{C} \to \B$  satisfying  $g_{cd}(c) \ne g_{cd}(d)$.  
Hence~(2) holds, with $m= |C\times C {\setminus} \{(c,c) \mid c \in C\}|$.

Suppose now that (2) holds. Let $\phi_i = \D(f_i)$ for $i =1,\ldots, m$. Then each $\im(\phi_i)$ is contained in $\D(\alg{C})$.  Let $\Z$ be the substructure $\langle \im (\phi_1 ) \cup \ldots \cup \im (\phi_m)\rangle$ of~$\D(\alg{C})$.  Since the duality is full,  there exists $\alg{C}' \in \ISP(\M)$ such that $\D(\alg{C'})\cong \Z$. For each~$i$, 
\[
\D(\B) \overset{\phi_i}{\longrightarrow}  \D(\alg{C}') \overset{\iota}{\hookrightarrow}  \D(\alg{C}),
\]
where~$\iota$ embeds $\D(\alg{C'})$ into~$\D(\alg{C})$. 
Applying the functor $\E$, and making use 
once again 
of the fullness of a strong duality, 
we see that $\alg{C}'$ is 
a homomorphic image of~$\alg{C}$ 
and that each~$f_i$ factors through~$\alg{C}'$.
If~$\Z$ were strictly contained in $\D(\alg{C})$, then Proposition~\ref{prop:cong-dualchar} 
would imply that the associated congruence $\theta_{\Z}$ is non-trivial and
contained in~$\bigcap\ker f_i$, 
contrary to our assumption.
Hence (2) implies (3). 
The converse is obtained essentially by reversing this argument.
\end{proof}

We now combine  Propositions~\ref{prop:cong-dualchar} and~\ref{lem:Dist}   
 to describe in dual terms
the lattice  $\Cong_{\ISP(\B)}(\A)$ for finite $\A,\B\in\CA$.
For this we first need a definition.  
Given finite structures $\X,\Y \in \IScP (\twiddle{\M})$ and a substructure $\Z$ of $\X$, 
let us say that $\str{Z}$ is a \defn{$\Y$-substructure} of $\X$ if there exist morphisms 
$\phi _1,\ldots \phi_m\colon\Y\to \X$ such that 
$\Z=\langle \im(\phi_1)\cup \cdots\cup\im(\phi_m)\rangle$. 

\begin{prop} \label{prop:DualRelCong} 
Let $\A, \B\in\CA$ be finite algebras and let $\Q=\ISP(\B)$.
Then the dual  order-isomorphism between $\Cong (\A)$ and the substructures of $\D(\A)$ restricts to a dual order-isomorphism between the lattice $ \Cong_{\Q}(\A)$ and the subfamily of $\D(\B)$-substructures of $\D(\A)$.   
Moreover these substructures form a lattice
which is  a join
subsemilattice of the lattice of all substructures of $\D(\A)$. 
\end{prop}

\begin{proof}  
We claim that a substructure $\Z$ of~$\D(\A)$ is a $\D(\B)$-substructure of~$\D(\A)$ 
if and only if $\A/\theta_\Z \in\Q$. 
The right-to-left direction follows by applying 
Proposition~\ref{lem:Dist},  (1) implies (3),  with  $\alg{C}$ as $\A/\theta_\Z \in \Q$.
For the other direction, suppose that $\Z$ is a $\D(\B)$-substructure of~$\D(\A)$.   
In particular, $\Z$ is a substructure of $\D(\A)$ and there exists some 
$\alg{C} \in \CA$ for which $\Z = \D(\alg{C})$.  But then there exists a surjective homomorphism $f$  from $\A$ onto~$\alg{C}$, so  $\A/\theta_\Z \in \CA$.
   Moreover,
since $\Z$ is a $\D(\B)$-substructure, 
there exist morphisms $\phi_i \colon \D(\B) \to \D(\A)$ for which
$\Z = \langle\im (\phi_1) \cup  \cdots \cup \im ( \phi_m)\rangle$.  
Observe that $\im (\phi_i) \subseteq \D(\alg{C})$. An application of Proposition~\ref{lem:Dist},  (3) implies (1),  with  $\alg{C}$ as~$\A/\theta_\Z$ yields $\A/\theta_\Z \in \Q$.

The final assertion follows from  the fact that $\Cong_{\Q}(\A)$ is  
a meet subsemilattice of $\Cong(\A)$.
\end{proof}

Finally, the following consequences of Proposition~\ref{prop:DualRelCong} 
constitute our main tool for obtaining the dual spaces of $\ISP(\B)$-subdirectly irreducible algebras.

\begin{coro}\label{cor:DualSIrr1}
Let $\B\in\CA$ be a finite algebra and let $\Q=\ISP(\B)$. If $\X$ is join-irreducible in the lattice of 
$\D(\B)$-substructures of $\D(\B)$, then $\E(\X)$ is $\Q$-subdirectly irreducible. 
\end{coro}
\begin{proof}
It follows from Proposition~\ref{prop:DualRelCong} that $\E(\X)$ is $\Q$-subdirectly irreducible if and only if $\X$ is join-irreducible in 
the lattice of $\D(\B)$-substructures of $\X$. Since  $\Y$ is in the lattice of $\D(\B)$-substructures of $\X$ if and only if $Y\subseteq X$  and $\Y$ is a $\D(\B)$-substructure of $\D(\B)$, the result follows.
\end{proof}

Given a finite  structure $\X\in \IScP(\twiddle{\M})$, let $\mathcal{S}_\X$ denote the lattice of $\X$-substructures of $\X$.

\begin{coro}\label{cor:DualSIrr2}
Let $\B\in\CA$ be a finite algebra and let $\Q=\ISP(\B)$. Then 
\[\Q=\ISP(\{\E(\Y)\mid \Y \mbox{ is maximal join-irreducible in $\mathcal{S}_{\D(\B)}$}\}).\] 
\end{coro}  
\begin{proof}
Since $\E(\X)\in\Q$ for each $\D(\B)$-substructure $\X$ of $\D(\B)$, we only need to prove that ${\B\in\Q':=\ISP(\{\E(\Y)\mid \Y \mbox{ is maximal join-irreducible in $\mathcal{S}_{\D(\B)}$}\})}$.

Let $\X_1,\ldots,\X_n$ be the set of maximal join-irreducible elements in $\mathcal{S}_{\D(\B)}$. 
Since $\D(\B)$ is itself a $\D(\B)$-substructure, it follows that $\D(\B)=\langle \X_1\cup\cdots\cup\X_n\rangle$. Hence the congruence $\theta_{\D(\B)}$ is equal to $\theta_{\X_1}\cap\cdots\cap\theta_{\X_n}$. Finally, observe that if $\theta_{\D(\B)} = \Delta_\B$, then $\B$ embeds into $\E(\X_1)\times\cdots\times\E(\X_n)\in\Q'$.
\end{proof}

Using these results we obtain a dual version of \proc{MinGenSet} applied to $\Q=\ISP(\B)$ with $\B\in \CA$. 
Letting $\X=\D(\B)$, proceed as follows:

\begin{enumerate}
\item[1.] \texttt{Determine the set $\mathcal{S}_\X$.}
\item[2.] \texttt{Calculate the set $\mathcal{V}$  of maximal  join-irreducible elements of $\mathcal{S}_{\X}$.}
\item[3.]  \texttt{Repeatedly remove from $\mathcal{V}$ any structure that is a morphic image of another structure in $\mathcal{V}$.}
\end{enumerate}

Combining Theorem~\ref{t:MinGenSet} and Corollaries~\ref{cor:DualSIrr1} and~\ref{cor:DualSIrr2}, we obtain that the set $\{\E(\Y)\mid \Y\in\mathcal{V}\}$ is the minimal set of generators for $\Q$.


\section{Dual formulation of \proc{SubPreHom}}\label{Sec:SPH}

Let us assume now and for the rest of the paper that $\M$ 
is generated by $s$ elements and no fewer. 
The aim of the algorithm $\proc{SubPreHom}$ is to provide an 
algebra
 $\A$ that is a subalgebra of $\F_\M(s)$ and has $\M$ as a homomorphic image: 
 that is, we seek an embedding $i \colon \A \to \F_\M(s)$ and a surjective homomorphism 
 $h \colon \A \to \M$.  
 In symbols: 
\[
\M \xdoubleheadleftarrow{\ h \ }  \A  \xhookrightarrow{\ i \ } \F_\M(s). 
\]
Note that any $\A\in \CA$ which has $\M$ as a homomorphic image has an
$s$-generated subalgebra $\B$ which also has $\M$ as a homomorphic image. 
Hence we may assume without loss of generality that $\A$ is a quotient of $\F_{\M}(s)$. 

Our first goal will be to describe the set of algebras~$\A \in \CA$ with these properties using 
the strong duality for $\CA$. 
Suppose that $\X$ is a substructure of $\twiddle{\M}^s$ and that there exist  a 
surjective morphism $\gamma\colon \twiddle{\M}^s \to \X$ and an 
embedding $\eta\colon\D(\M)\to \X$. 
In symbols:
\[
  \D(\M) \xhookrightarrow{\ \eta \ }\alg{X} \xdoubleheadleftarrow{\ \gamma\ } \twiddle{\M}^s.
\]
We will refer to the triple $(\X,\gamma,\eta)$ as a \defn{Test Space configuration}, or \defn{{\rm TS}-configuration} for short. 

\begin{prop}\label{Prop:TSMJustification}
Let  $(\X,\gamma,\eta)$ be a Test Space configuration. Then  $\ISP(\F_\M(s))= \ISP(\E(\X))$ and the following are equivalent:
\begin{enumerate}
\item[{\rm (1)}]		$\Sigma \imp \f\eq\p$ is admissible in $\CA$;
\item[{\rm (2)}]		$\Sigma \mdl{\E(\X)} \f\eq\p$.
\end{enumerate}
\end{prop}
\begin{proof}   
Since $\gamma\colon \twiddle{\M}^s \to \X$ is surjective, 
it is an epimorphism in $\IScP(\twiddle{\M})$. 
It follows that $\E(\gamma)\colon \E(\X)\to \E(\twiddle{\M}^s)$ is a monomorphism in $\ISP(\M)$. 
Since $\ISP(\M)$ is a quasivariety, $\E(\gamma)$ is an embedding 
from $\E(\X)$ into $\E(\twiddle{\M}^s)$.  
As observed above, the dual space of $\F_\M(s)$ is
isomorphic to $\twiddle{\M}^s$, so $\E(\X)\in \ope{IS}(\F_\M(s))$.  
Hence $\ISP(\E(\X)) \subseteq\ISP(\F_\M(s))$. Since the duality yielded by $\twiddle{\M}$ is strong, 
the homomorphism $\E(\eta) \colon\E(\X)\to \E(\D(\M))$ is surjective. 
But also $\E(\D(\M)) \cong \M$, so
$\HSP(\M) \subseteq \HSP(\E(\X)) \subseteq \HSP(\ISP(\F_\M(s))) = \HSP(\F_\M(s)) = \HSP(\M)$. 
Hence $\F_\M(s) = \F_{\E(\X)}(s) \in \ISP(\E(\X))$, 
and we obtain $\ISP(\F_\M(s))\subseteq \ISP(\E(\X))$. 

The equivalence of (1) and (2) now follows from \eqref{Eq:AdmFree} in 
Section~\ref{sec:admiss}. 
\end{proof}


\section{The Test Spaces Method}\label{Sec:TestSpace}

In this section we use the results of Sections~\ref{Sec:MGS} and~\ref{Sec:SPH} 
to provide a procedure that produces for any finite algebra $\M$, a minimal set of algebras 
for checking admissibility in the quasivariety generated by $\M$. 
For reasons that will soon become apparent, we call this procedure the {\em Test Spaces Method}.  

The Test Spaces Method, which combines the dual formulations of \proc{SubPreHom} and $\proc{MinGenSet}$, is presented below. 
\medskip

\begin{center}
\begin{tabular}{ll}
\hline
\multicolumn{2}{c}{Test Spaces Method}\\
\hline \\
0.&\texttt{Find $\twiddle{\M}$ that yields a strong duality for $\CA=\ISP(\M)$.}\\[.025in]
1. & \texttt{Compute $\D(\M)$.}\\[.025in]
2. &\texttt{Find a TS-configuration $(\X,\gamma,\eta)$ with~$X$ of minimum size.}  \\[.025in]
3. &\texttt{Determine the set $\mathcal{M}$ of maximal join-irreducible elements of $\mathcal{S}_{\X}$}.\\[.025in]
4. & \texttt{Construct a set $\mathcal{V}$ by repeatedly removing from $\mathcal{M}$ any structure} \\
& \texttt{that is a morphic image of some other structure in the set.}\\[.025in]
5.& \texttt{Compute  $\K=\{\,E(\X)\mid\X\in \mathcal{V}\,\}$.}\\[.1in]
\end{tabular}
\end{center}
Step~0  is so labelled because in many cases a suitable duality can be found in the literature; indeed, 
we assume that we already have such a duality to hand. 
Steps~1--4  then form the core of the method, corresponding to the
dualised versions of  the algorithms $\proc{MinGenSet}$ and $\proc{SubPreHom}$.

For Step~1, we compute the dual space $\D(\M)$.  
This is $\End (\M)$, the set of endomorphisms of~$\M$, 
with the operations (and partial operations if any) and relations defined pointwise. 
In particular, the action on $\End(\M)$ of any unary operation $g$ in~$G$ is by composition;  
here we think of $g$ as having codomain $\MT$. 
To simplify notation, we write an $n$-tuple  $(x_1,\ldots, x_n)$ as $x_1\cdots x_n$. 

For Step~2, we calculate a minimal TS-configuration $(\X,\gamma,\eta)$ using 
the requirements that $\X$ must contain a copy of $\D(\M)$ and be a morphic image of 
$\MT^s$, the dual space of $\F_{\CA}(s)$.
It is not strictly necessary here to obtain the smallest $\X$; in particular, we could always use $\X=\MT^s$. 
However, any reduction in the size of $\X$ will greatly simplify the process of 
calculating $\mathcal{S}_{\X}$. 
In all the case studies presented in the next section, $\X$ is of 
minimal size and indeed the only maximal join-irreducible of $\mathcal{S}_{\X}$ 
happens to be $\X$ itself. In other cases it might be more practical to choose a TS-configuration $\X$ that is not necessarily of minimal size but is sufficiently small for $\mathcal{S}_{\X}$ to be calculated.
 
Step~3 requires us to determine the set $\mathcal{M}$ of join-irreducible elements of $\mathcal{S}_\X$. 
 In all the examples considered in this paper,  $\X$ itself has a unique lower cover 
 in the lattice $\mathcal{S}_{\X}$; 
this is because the particular form taken by the alter ego $\MT$ constrains 
the possible morphisms from $\X$ to $\X$. In this situation, $\mathcal{M}$ 
is just $\{\X\}$ and Step~4 can be skipped; we shall do this henceforth in our case 
studies without explicit comment. Step~5 then becomes the calculation of $\E(\X)$, 
now with the assurance our theory provides that this is indeed the minimal generator 
for $\ISP(\F_{\M}(s))$.

\begin{DMex}  \label{DMex2}  
Recall from part 1 of this running example that the algebra $\alg{D_4}$ generating 
the (quasi)variety $\class{DM}$ of De Morgan algebras is $2$-generated. 
We therefore apply the Test Spaces Method with $s=2$, using the natural duality 
for $\class{DM}$ described in part 2.

\begin{enumerate} 

\item[1.] \texttt{Compute $\D(\alg{D_4})$.} 

The universe of  $\D(\alg{D_4})$ is the set $\{e_1,e_2\}$ of endomorphisms of $\alg{D_4}$, 
where $e_1$ is the identity map and $e_2$ exchanges~$a$ and~$b$. 
Observe that $e_1(b)=b\preccurlyeq a= e_2(b)$ and $e_2(a)=b\preccurlyeq a=e_1(a)$. 
Hence $e_1$ and $e_2$ are incomparable in $(\D(\alg{D_4}); \preccurlyeq)$. 
The unary operation~$g$ acts on $e_1$ and~$e_2$ by composition, 
giving $g(e_1)=e_2$ and $g(e_2)=e_1$.

\item[2.]  \texttt{Find a  TS-configuration $(\X,\gamma,\eta)$ with~$X$ of minimum size.}  

In order for $\D(\alg{D_4})$ to embed into $\X=(X; g, \preccurlyeq)$, the latter must contain 
incomparable elements  $u,v$ satisfying  $g(u)=v$ and $g(v)=u$. 
Since the poset $\twiddle{\alg{D_4}}^2$ has top and bottom elements, 
$X$ must also have top and  bottom elements $\top$ and $\bot$, respectively. 
Moreover, there exist elements of  $\twiddle{\alg{D_4}}^2$ fixed by $g$, 
and hence some element of $\X$ is fixed by~$g$; this element cannot be $u$, $v$, $\top$, or $\bot$. 
So~$|X| \geq 5$. A natural candidate for the universe of~$\X$ is then the subset 
$X=\{ab,ba,aa,bb,00\}$ of $\twiddle{\alg{D_4}}^2$ in Fig.~\ref{fig:XforDM}, 
taking $u=ab$, $v=ba$, $\top=aa$, $\bot=bb$, $w=00$.

\begin{figure}[ht]
\begin{center}
	\begin{tikzpicture}[scale=.8]  

		\node [label=above:{$\top$}](top) at (0,3) {};   
		\node [label=right:{$v$}](v) at (0,1.5) {}; 
		\node [label=below:{$\bot$}](bot) at (0,0) {}; 
		\node [label=left:{$u$}](u) at (-1.5,1.5) {}; 
		\node  (w) at (1.5,1.5) {}; 
		\node [xshift=16pt] at (w) {$w$};
		\fill (top) circle (2pt);
		\fill  (bot) circle (2pt);
		\fill  (u)  circle (2pt);
		\fill  (v) circle (2pt);
		\fill  (w) circle (2pt);
		
		\draw [<->]  (top) .. controls +(-3,.5) and +(-3,-.5) .. (bot);
		\draw [<->]  (u) -- (v);
		\draw [<-]  (w) .. controls +(.5,.4) and +(.5,-.4) .. (w);

		\draw [shorten <=-2pt, shorten >=-2pt] (u) -- (top);
		\draw [shorten <=-2pt, shorten >=-2pt] (v) -- (top);
		\draw [shorten <=-2pt, shorten >=-2pt] (w) -- (top);
		\draw [shorten <=-2pt, shorten >=-2pt] (bot) -- (u);
		\draw [shorten <=-2pt, shorten >=-2pt] (bot) -- (v);
		\draw [shorten <=-2pt, shorten >=-2pt] (bot) -- (w);

	\end{tikzpicture}
\caption{Test space $\X$ for De Morgan algebras  \label{fig:XforDM}}
\end{center}
\end{figure}
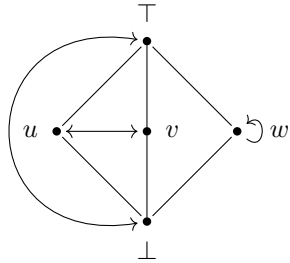

It is easy to see that $X$ is closed under the action of  $g$ and that the 
substructure~$\X$ of $\twiddle{\alg{D_4}}^2$ determined by~$X$ 
satisfies the required conditions. 
We define $\eta\colon\D(\alg{D_4}) \to \X$ by $\eta(e_1)=u$ and $\eta(e_2)=v$. 
Then $\eta$ is an order isomorphism and preserves the $g$-action. 
Fig.~\ref{fig:TSforDM} shows that there is an obvious partition of $\twiddle{\alg{D_4}}^2$ 
compatible with the structure such that each set of the partition contains precisely  
one point of~$X$. The associated quotient map is therefore a morphism from 
$\twiddle{\alg{D_4}}^2$ onto~$\X$, as shown in Fig.~\ref{fig:XforDM}.

\begin{figure}[ht]
\begin{tikzpicture}[scale=1]  
	\node [label=above:{$aa$}] (aa) at (7,4) {};
     \node (a0) at (4,3) {};
	\node (a1) at (6,3)  {};
     \node (0a) at  (8,3) {};
	\node (1a) at (10,3) {};    
	\node (0b) at (4,1) {};
	\node (b0) at (10,1) {};
	\node (b1) at (8,1) {};
	\node (1b) at (6,1) {};
	\node [label=below:{$bb$}] (bb) at (7,0) {};	
	\node [label=right:{$ba$}] (ba) at (12,2) {};   
	\node [label=left:{$ab$}] (ab) at (2,2) {};    
	\node [label=left:{$00$}] (00) at (4,2) {};
	\node [label=right:{$01$}] (01) at (8,2)   {};	
     \node [label=right:{$10$}] (10) at (10,2) {};    
	\node [label=left:{$11$}] (11) at (6,2)  {};  
	
	\node [xshift=-13pt, yshift=3pt] at (a0) {$a0$};
	\node [xshift=-13pt, yshift=3pt] at (a1)  {$a1$};
     \node [xshift=13pt, yshift=3pt] at (0a) {$0a$};
	\node [xshift=13pt, yshift=3pt] at (1a)  {$1a$};    
	\node [xshift=-13pt, yshift=-3pt] at (0b) {$0b$};
	\node [xshift=13pt, yshift=-3pt] at (b0)  {$b0$};
     \node [xshift=13pt, yshift=-3pt] at (b1) {$b1$};
	\node [xshift=-13pt, yshift=-3pt] at (1b)  {$1b$};    

	\fill (aa) circle [radius=2pt];
     \draw (a0) circle [radius=2pt];
	\draw (a1) circle [radius=2pt];
     \draw (0a) circle [radius=2pt];
	\draw (1a) circle [radius=2pt];
	\draw (0b) circle [radius=2pt];
	\draw (b0) circle [radius=2pt];
	\draw  (b1) circle [radius=2pt];
	\draw  (1b) circle [radius=2pt];
	\fill  (bb) circle [radius=2pt];

	\fill (ba) circle [radius=2pt];
	\fill (ab) circle [radius=2pt];
	\fill (00) circle [radius=2pt];
	\draw (01) circle [radius=2pt];
     \draw (10) circle [radius=2pt];
	\draw (11) circle [radius=2pt];

	\draw [shorten <=-2pt, shorten >=-2pt] (aa) -- (0a);
	\draw [shorten <=-1pt, shorten >=-1pt] (aa) -- (a0);
	\draw [shorten <=-2pt, shorten >=-2pt] (aa) -- (a1);
	\draw [shorten <=-1pt, shorten >=-1pt] (aa) -- (1a);
	\draw [shorten <=-1pt, shorten >=-1pt] (0a) -- (00);
	\draw [shorten <=-1pt, shorten >=-1pt] (0a) -- (01);
	\draw [shorten <=-1pt, shorten >=-1pt] (0a) -- (ba);
	\draw [shorten <=-1pt, shorten >=-1pt] (a0) -- (00);
	\draw [shorten <=-1pt, shorten >=-1pt] (a0) -- (ab);
	\draw [shorten <=-1pt, shorten >=-1pt] (a0) -- (10);
	\draw [shorten <=-1pt, shorten >=-1pt] (a1) -- (01);
	\draw [shorten <=-1pt, shorten >=-1pt] (a1) -- (ab);
   	\draw [shorten <=-1pt, shorten >=-1pt] (a1) -- (11);
	\draw [shorten <=-1pt, shorten >=-1pt] (1a) -- (ba);
	\draw [shorten <=-1pt, shorten >=-1pt] (1a) -- (10);
	\draw [shorten <=-1pt, shorten >=-1pt] (1a) -- (11);
  	\draw [shorten <=-1pt, shorten >=-1pt] (0b) -- (00);
	\draw [shorten <=-1pt, shorten >=-1pt] (0b) -- (01);
	\draw [shorten <=-1pt, shorten >=-1pt] (0b) -- (ab);
	\draw [shorten <=-1pt, shorten >=-1pt] (b0) -- (00);
	\draw [shorten <=-1pt, shorten >=-1pt] (b0) -- (ba);
	\draw [shorten <=-1pt, shorten >=-1pt] (b0) -- (10);
	\draw [shorten <=-1pt, shorten >=-1pt] (b1) -- (01);
	\draw [shorten <=-1pt, shorten >=-1pt] (b1) -- (ba);
 	\draw [shorten <=-1pt, shorten >=-1pt] (b1) -- (11);
	\draw [shorten <=-1pt, shorten >=-1pt] (1b) -- (ab);
	\draw [shorten <=-1pt, shorten >=-1pt] (1b) -- (10);
	\draw [shorten <=-1pt, shorten >=-1pt] (1b) -- (11);
   	\draw [shorten <=-1pt, shorten >=-1pt] (bb) -- (0b);
	\draw [shorten <=-1pt, shorten >=-1pt] (bb) -- (b0);
	\draw [shorten <=-2pt, shorten >=-2pt] (bb) -- (b1);
	\draw [shorten <=-2pt, shorten >=-2pt] (bb) -- (1b);

	\draw[gray] (3.2,2.7) -- (3.2,4.6) -- (10.8,4.6) -- (10.8, 2.7) -- (3.2,2.7) ; 
	\draw[gray] (3.2,-0.7) -- (3.2,1.3) -- (10.8,1.3) -- (10.8, -0.7) -- (3.2,-0.7) ; 
	\draw[gray] (3.2,1.8) -- (3.2,2.3) -- (10.8,2.3) -- (10.8, 1.8) -- (3.2,1.8) ; 
	\draw[gray] (1.2,1.8) -- (1.2,2.3) -- (2.5,2.3) -- (2.5, 1.8) -- (1.2,1.8) ; 
	\draw[gray] (11.5,1.8) -- (11.5,2.3) -- (12.8,2.3) -- (12.8, 1.8) -- (11.5,1.8) ; 

	\end{tikzpicture}
\caption{Setting up  a TS-configuration  for De Morgan algebras}\label{fig:TSforDM}
\end{figure}
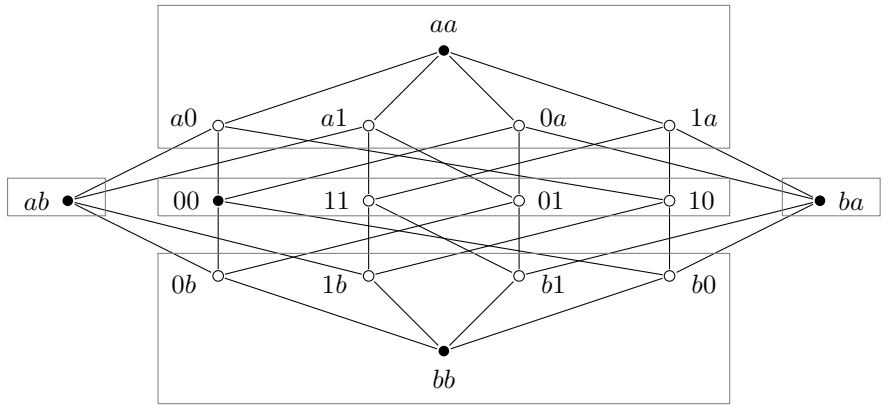

\item[3.] \texttt{Determine the set $\mathcal{M}$ of maximal join-irreducible elements of $\mathcal{S}_{\X}$.}

We find that  the subset of $\mathcal{S}_{\X}$ of universes of images of morphisms from $\X$ to itself is $\{\{w\},\{\bot,w,\top\},X\}$  (note that any such image must contain a point fixed by~$g$).  It follows that $\mathcal{S}_{\X}$ is precisely $\{\{w\},\{\bot,w,\top\},X\}$, so $\X$ is itself $\X$-join-irreducible.

\item[5.] \texttt{Compute  $\E(\X)$.}

We find either directly or, more quickly, using the restricted Priestley duality,
 that $\E(\X)$ is $\overline{\alg{D_{42}}}$, the algebra produced by {\tafa} 
 (see part 1 of this running example).

\end{enumerate}
\end{DMex}


\section{Case studies}\label{Sec:CaseStudies}

In this section we present applications of the Test Spaces Method. Our aim is to focus on the outcomes, highlighting the computational challenges of using either free algebras directly or the algebraic approach of~\cite{MR13} to check admissibility. We therefore select examples for which the sizes of the free algebras increases very rapidly as the size of the generating algebra increases. Indeed, all of the case studies  presented in this section involve free algebras that are too large to be handled by {\tafa}.

The (quasi)varieties considered here are all (at least term-equivalent to) finitely generated (quasi)varieties of bounded distributive lattices equipped with one or two unary operations, each of which is an endomorphism or a dual endomorphism. Our examples are accordingly of similar algebraic type to those for which admissibility can be investigated with the aid of {\tafa}, including De Morgan algebras, Kleene algebras, and Stone algebras. This similarity allows us to focus on revealing the challenges inherent in working with larger generating algebras, without different factors coming into play. We have chosen in particular to study the variety of MS-algebras. This class of algebras was  introduced  by Blyth and Varlet as a common generalisation of De Morgan and Stone algebras; see~\cite{BV}.  This variety and the larger variety of Ockham algebras have been extensively studied in their own right.  More importantly, duality tools applied first to these special varieties have pointed the way to major advances in the wider theory of natural dualities, on some of which we tacitly rely.   Moreover, amenable strong dualities are available in the literature for $\text{MS}$-algebras and its subvarieties, and for the related examples we consider, \textit{viz.}  double Stone algebras, Kleene--Stone algebras, and involutive Stone algebras.  

We treat the dualities in black-box fashion. All other steps in the Test Space Method can then in principle be carried out automatically (that is, implemented as a terminating computer program); however, since we carry out these steps by hand, we often take advantage of the theory of natural dualities to simplify calculations.

Table~\ref{table:casestud} summarises the results obtained in this section. 
In all these examples, the algebra $\M$ is $2$-generated, so it suffices to consider free algebras 
on two generators. 

\begin{table}[ht]
\begin{center}
\begin{tabular}{lrrrr}
  	$\ISP(\M)$ & $|M|$  & $ |\F_\M(2)|$\quad  & $|X|$ &  $|\E(\X)|$   \\
	\hline &&&&\\[-.3cm] 
	De Morgan algebras & $4$ \ & $168$ \ & $5$ \ & $10$ \ \\[.1cm]   
   MS-algebras & $6$ \ & $8\thinspace790$ \  & $6$ \ &  $14$ \ \\[.1cm]
  subvarieties of  MS-algebras: \\[.1cm] 
	\quad     $\class{K}_2$   & $4$ \ & $414$ \ &$4$ \ & $7$ \ \\[.1cm]  
	\quad     $\class{K}_3$   & $5$ \ & $3\thinspace059$ \ &   $4$ \ &  $9$ \ \\[.1cm]   
   double Stone algebras  & $4$ \ & $7\thinspace776$ \ & $4$ \ & $8$ \ \\[.1cm]  
 involutive Stone algebras & $6$ \ & \ $3\thinspace483\thinspace648$ \  & $6$ \ & $20$ \  \\[.1cm]
  	Kleene--Stone algebras & $5$ \ & $1\thinspace741\thinspace824$ \  & $4$ \ & $12$ \ \\[.1cm]
\end{tabular}
\end{center}
\caption{Case studies}\label{table:casestud}
\end{table}


\begin{MSex}
Fig.~\ref{fig:MS} depicts both the $2$-generated algebra 
\[
\alg{MS} =( \{0, a, b, c, d, 1\}; \wedge, \vee, f, 0, 1)
\]
generating the (quasi)variety $\class{MS} = \ISP(\alg{MS})$ of MS-algebras and an alter ego yielding a 
strong  duality $\twiddle{\alg{MS}}= (\{ 0,a,b,c,d,1\}; g,\preccurlyeq)$ (see~\cite{UnOpe}).
\begin{figure} [ht]
\begin{center}
	\begin{tikzpicture}[scale=1]  
	  	\node [label=above:$1$] (1) at (-3,2) {};
		\node [label=above:$d$] (d) at (-4,1)  {};
  		\node [label=right:$c$] (c) at (-2,1) {};
		\node [label=below:$b$] (b) at (-3,0)  {};
  		\node [label=below:$a$] (a) at (-1,0) {};
		\node [label=below:$0$] (0) at (-2,-1)  {};

		\draw [thick,shorten <=-1pt, shorten >=-1pt] (1) -- (c);
		\draw [thick,shorten <=-1pt, shorten >=-1pt] (1) -- (d);
		\draw [thick,shorten <=-1pt, shorten >=-1pt] (b) -- (c);
		\draw [thick,shorten <=-1pt, shorten >=-1pt] (b) -- (d);
		\draw [thick,shorten <=-1pt, shorten >=-1pt] (c) -- (a);
		\draw [thick,shorten <=-1pt, shorten >=-1pt] (b) -- (0);
		\draw [thick,shorten <=-1pt, shorten >=-1pt] (0) -- (a);

		\draw [thin,<->]  (1) .. controls +(-3,-0.2) and +(-2,.2) .. (0);
		\draw [thin,->]  (b) .. controls +(-.5,.1) and +(0,-.5) .. (d);
		\draw [thin,->]  (a) .. controls +(.5,-.4) and +(.5,.4) .. (a);
		\draw [thin,->]  (d) .. controls +(-.5,-.4) and +(-.5,.4) .. (d);
		\draw [thin, ->]  (c) -- (0);

		\draw (1) circle [radius=2pt];
		\draw (c) circle [radius=2pt];
		\draw (d) circle [radius=2pt];
		\draw (b) circle [radius=2pt];
		\draw (a) circle [radius=2pt];
		\draw (0) circle [radius=2pt];
		

 	  	\node [label=above:$d$] (a1) at (3,2) {};
		\node [label=above:$1$] (ad) at (2,1)  {};
  		\node [label=right:$b$] (ac) at (4,1) {};
		\node [label=below:$c$] (ab) at (3,0)  {};
  		\node [label=below:$0$] (aa) at (5,0) {};
		\node [label=below:$a$] (a0) at (4,-1)  {};

		\draw [thick,shorten <=-1pt, shorten >=-1pt] (a1) -- (ac);
		\draw [thick, shorten <=-1pt, shorten >=-1pt] (a1) -- (ad);
		\draw [thick,shorten <=-1pt, shorten >=-1pt] (ab) -- (ac);
		\draw [thick,shorten <=-1pt, shorten >=-1pt] (ab) -- (ad);
		\draw [thick,shorten <=-1pt, shorten >=-1pt] (ac) -- (aa);
		\draw [thick,shorten <=-1pt, shorten >=-1pt] (ab) -- (a0);
		\draw [thick,shorten <=-1pt, shorten >=-1pt] (a0) -- (aa);

		\draw [thin,<->]  (a1) .. controls +(-3,-0.2) and +(-2,.2) .. (a0);
		\draw [thin,->]  (ab) .. controls +(-.5,.1) and +(0,-.5) .. (ad);
		\draw [thin,->]  (aa) .. controls +(.5,-.4) and +(.5,.4) .. (aa);
		\draw [thin,->]  (ad) .. controls +(-.5,-.4) and +(-.5,.4) .. (ad);
		\draw [thin,->]  (ac) -- (a0);

		\draw[thin] (a1) circle [radius=2pt];
		\draw[thin]  (ac) circle [radius=2pt];
		\draw[thin] (ad) circle [radius=2pt];
		\draw[thin] (ab) circle [radius=2pt];
		\draw[thin]  (aa) circle [radius=2pt];
		\draw[thin] (a0) circle [radius=2pt];

	\end{tikzpicture}
\end{center}
\caption{The generating algebra $\alg{MS}$ and its alter ego}\label{fig:MS}
\end{figure}
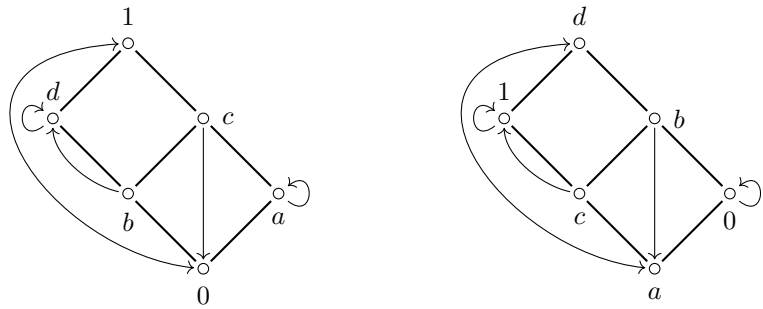

Since the strong duality given by $\twiddle{\alg{MS}}$ 
coincides with the restricted Priestley duality, the lattice reduct of the free algebra 
$\F_{\alg{MS}}(2)$ is isomorphic to the lattice of up-sets of 
$\boldsymbol 2^2 \times \boldsymbol 3^2$, which has $8\thinspace 790$ 
elements and is much too large to be handled by {\tafa}.

The Test Spaces Method for this quasivariety proceeds as follows:

\begin{enumerate}

\item[1.]  \texttt{Compute $\D(\alg{MS})$.}

Let us denote an endomorphism $e$ of $\alg{MS}$ by the sextuple 
$\big(e(0), e(a), e(b), e(c), e(d), e(1)\big)$. 
There are three such endomorphisms: the identity map,  $\id$, and the maps~$e_1$ and 
$e_2$, given respectively by the sextuples $(0,d, a, 1, a, 1)$ and $(0, a,d,1,d,1)$. 
Equipping  $\{ \id, e_1,e_2\}$ with  the pointwise order from
 $(\twiddle{\alg{MS}};\preccurlyeq)$, we see that $\id \prec e_2$ and 
 $e_1$ is incomparable with both $\id$ and $e_2$.

The space $\D(\alg{MS})$ is therefore as depicted in Fig.~\ref{fig:DualMS}.
\begin{figure} [ht]
\begin{center}
	\begin{tikzpicture}[scale=.55]  
	  	\node [label=above: {$e_2$}] 
 (e2) at (0,1) {};
		\node [label=above:{$e_1$}] 
(e1) at (3,1)  {};
  		\node [label=below: {$\id$}]
 (Id) at (0,0) {};
			
		\draw [thick, shorten <=-1pt, shorten >=-1pt] (Id) -- (e2);

		\draw [thin, <->]  (e2) -- (e1);
		\draw [thin, ->]  (Id) .. controls +(1,.1) and +(-.5,-.5) .. (e1);

		\draw (e1) circle [radius=2pt];
		\draw (e2) circle [radius=2pt];
		\draw (Id) circle [radius=2pt];

	\end{tikzpicture}
\end{center}
\caption{ $\D(\alg{MS}) $ }\label{fig:DualMS}
\end{figure}

\item[2.]  \texttt{Determine a TS-configuration $(\X,\gamma,\eta )$ with~$X$ of minimum size.}  

\noindent Let $\X$ be the substructure of $\twiddle{\alg{MS}}^2$   with universe $X=\{ 00, aa, ad, da, ba, dd\}$ (depicted in Fig.~\ref{fig:MS2}). Then $\X$ is a retract of $\twiddle{\alg{MS}}^2$ given by the morphism
\[
\eta(x)=\begin{cases}
x & \mbox{if }x\in    \{ ad, da, ba\},\\ 
dd & \mbox{if } x\in 
            \{b,d\}^2 \cup 
            (\{b,d\} \times\{ 0,c,1\})  \cup  
                                 (\{0,1,c\} \times \{b,d\} ),  
                                 
\\
aa & \mbox{if } x\in 
                           (\{a\} \times \{0,c,1\}) \cup 
                           (\{ 0,c,1\} \times \{a\})  \cup \{ ab\}  ,\\
00 & \mbox{if }  x\in \{  0,c,1\}^2,
\end{cases}
\]
and the substructure of $\X$ determined by $\{ ad, da, ba\}$  is isomorphic to $\D(\alg{MS})$.

\begin{figure} [ht]
\begin{center}
	\begin{tikzpicture}[scale=.75]  
	  	\node [label=above:{$dd$}] (dd) at (1,2) {};
	  	\node [label=left:{$da$}] (da) at (1,1) {};
		\node [label=above:{$ad$}] (ad) at (3,1)  {};
  		\node [label=left:{$ba$}] (ba) at (1,0) {};
 		\node [label=below:{$00$}] (00) at (-1,1) {};
	  	\node [label=below:{$aa$}] (aa) at (1,-1) {};
			
		\draw [shorten <=-1pt, shorten >=-1pt] (ba) -- (da);
		\draw [shorten <=-1pt, shorten >=-1pt] (dd) -- (da);
		\draw [shorten <=-1pt, shorten >=-1pt] (dd) -- (ad);
		\draw [shorten <=-1pt, shorten >=-1pt] (dd) -- (00);
		\draw [shorten <=-1pt, shorten >=-1pt] (aa) -- (00);
		\draw [shorten <=-1pt, shorten >=-1pt] (ba) -- (aa);
		\draw [shorten <=-1pt, shorten >=-1pt] (ad) -- (aa);

		\draw [<->,dashed]  (da) -- (ad);
		\draw [->,dashed]  (ba) -- (ad);
		\draw [->,dashed]  (00) .. controls +(-.5,-.4) and +(-.5,.4) .. (00);
		\draw [<->,dashed]  (dd) .. controls +(-4,0) and +(-4,0) .. (aa);

		\draw (da) circle [radius=2pt];
		\draw (ad) circle [radius=2pt];
		\draw (ba) circle [radius=2pt];
		\draw (00) circle [radius=2pt];
		\draw (dd) circle [radius=2pt];
		\draw (aa) circle [radius=2pt];

	\end{tikzpicture}
\end{center}
\caption{Minimal $\X$  for $\text{MS}$-algebras }\label{fig:MS2}
\end{figure}
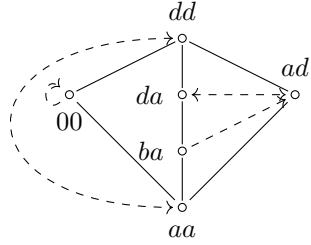

\item[3.] \texttt{Determine the set $\mathcal{M}$ of maximal join-irreducible elements of $\mathcal{S}_{\X}$.}

The $\X$-substructures of $\X$ are determined by the sets $\{00\}$, $\{\bot,00,\top\}$, $\{\bot,00,da,ad,\top\}$, and~$X$. Hence $\mathcal{M}= \{\X\}$.  
\item[5.] \texttt{Compute $\E(\X)$.}

Since we are dealing here with a natural duality 
which is also a restricted Priestley duality,  
the lattice reduct of $\E(\X)$ is isomorphic to the lattice of up-sets of $(X;\preccurlyeq)$, 
and hence to the product of a $3$-element chain and two $2$-element chains with extra top and 
bottom (see Fig.~\ref{fig:EXMS}). 

\begin{figure} [ht]
\begin{center}
	\begin{tikzpicture}[scale=.7]  

		\node (1) at (5,0) {}; 
		\node (2) at (5,3) { };   
		\node (3) at (5,2.25) {}; 
		\node (4) at (5,1.5) {}; 
		\node (5) at (5,.75) {}; 
		\node (6) at (4,2.25) {};  
		\node (7) at (4,1.5) {}; 
		\node (8) at (6,2.25) {}; 
		\node (9) at (6,1.5) {}; 
		\node (10) at (3,2.25) {}; 
		\node (11) at (3,3) {}; 
		\node (12) at (4,3) {}; 
		\node (13) at (4,3.75) {}; 
		\node (14) at (4,4.5) {}; 

		\draw (1) circle (2pt);
		\draw (2) circle (2pt);
		\draw  (3) circle (2pt);
		\draw  (4)  circle (2pt);
		\draw  (5) circle (2pt);
		\draw  (6) circle (2pt);
		\draw  (7) circle (2pt);
		\draw  (8) circle (2pt);
		\draw  (9) circle (2pt);
		\draw  (10) circle (2pt);
		\draw  (11) circle (2pt);
		\draw  (12) circle (2pt);
		\draw  (13) circle (2pt);
		\draw  (14) circle (2pt);

   		\draw [shorten <=-1pt, shorten >=-1pt] (1) -- (5);
		\draw [shorten <=-2pt, shorten >=-2pt] (5) -- (7);
		\draw [shorten <=-1pt, shorten >=-1pt] (5) -- (4);
		\draw [shorten <=-2pt, shorten >=-2pt] (5) -- (9);
		\draw [shorten <=-1pt, shorten >=-1pt] (7) -- (6);
		\draw [shorten <=-1pt, shorten >=-1pt] (9) -- (8);
		\draw [shorten <=-2pt, shorten >=-2pt] (2) -- (8);
		\draw [shorten <=-1pt, shorten >=-1pt] (2) -- (3);
		\draw [shorten <=-2pt, shorten >=-2pt] (2) -- (6);
		\draw [shorten <=-2pt, shorten >=-2pt] (2) -- (13);
		\draw [shorten <=-2pt, shorten >=-2pt] (7) -- (3);
		\draw [shorten <=-2pt, shorten >=-2pt] (9) -- (3);
		\draw [shorten <=-2pt, shorten >=-2pt] (6) -- (4);
		\draw [shorten <=-2pt, shorten >=-2pt] (8) -- (4);
		\draw [shorten <=-2pt, shorten >=-2pt] (3) -- (12);
		\draw [shorten <=-2pt, shorten >=-2pt] (11) -- (6);
		\draw [shorten <=-2pt, shorten >=-2pt] (7) -- (10);
		\draw [shorten <=-2pt, shorten >=-2pt] (12) -- (10);
		\draw [shorten <=-2pt, shorten >=-2pt] (13) -- (11);
		\draw [shorten <=-1pt, shorten >=-1pt] (12) -- (13);
		\draw [shorten <=-1pt, shorten >=-1pt] (10) -- (11);
		\draw [shorten <=-1pt, shorten >=-1pt] (14) -- (13);


\draw [<->,dashed]  (1) .. controls +(3,1) and +(3,-1) .. (14);
\draw [->,dashed]  (2) .. controls +(.5,-.9) and +(.5,.9) .. (5);
\draw [<->,dashed]  (13) .. controls +(-2.5,.5) and +(-2.5,-.5) .. (5);
\draw [->,dashed]  (3) -- (4);
\draw [->,dashed]  (6) -- (10);
\draw [->,dashed]  (7) -- (11);
\draw [<->,dashed]  (13) -- (5);
\draw [<->,dashed]  (12) -- (4);
\draw [<->,dashed]  (10) .. controls +(-.5,0) and +(-.5,0) .. (11);
\draw [<->,dashed]  (8) .. controls +(.5,0) and +(.5,0) .. (9);	\end{tikzpicture}
\end{center}
\caption{ $\E(\X)$ for MS-algebras}\label{fig:EXMS}
\end{figure}
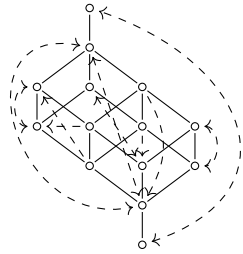

\end{enumerate}

\end{MSex}


\begin{Kiex}  \label{Kiex}
The variety of MS-algebras provided a good choice for a case study because it contains as subvarieties various classes to which {\tafa} had previously been applied.    Moreover, 
the structure of the lattice $\Lambda(\class{MS})$ of all subvarieties of $\class{MS}$ 
is well understood; see for example~\cite[Fig.~1]{AP94}. 
Since the varieties $\class{S}$, $\class{K}$, and $\class{DM}$ lie low down in  
$\Lambda(\class{MS})$, it is unsurprising that they are tractable by hand or with {\tafa}.
 It is natural therefore to ask what happens for proper subvarieties higher up in  
 the subvariety lattice.  
 We focus here on two varieties $\class{K}_2 = \ISP(\alg{K_2})$ and 
 $\class{K}_3= \ISP(\alg{K_3})$, where $\alg{K_2}$ and $\alg{K_3}$
  are $2$-generated subalgebras of $\alg{MS}$ with 
\[
K_2  = \{0,a,c,1\} \quad \text{and} \quad 
K_3 = \{0,b,c,d,1\}.
\]
These varieties both contain $\class{S}$ and $\class{K}$, but neither contains $\class{DM}$. 
There exist strong dualities for both $\class{K}_2$ and $\class{K}_3$, which can be 
used to apply the Test Spaces Method.   
However, the simplest way for a knowledgeable duality theorist to proceed is to use instead 
the techniques of multisorted duality theory which originated in~\cite{genpig}
 (see also~\cite[Chapter~7]{CD98}). 
Hence we shall simply present our conclusions without proof,
formulated for the single-sorted theory from Section~\ref{sec:natdual}. 

A strongly dualising alter ego for $\alg{K_2}$  is
\[
\twiddle{\alg{K_2}} =  (\{0,a,c,1\};  f,  \preccurlyeq ,r),
\]
where $f$ is the endomorphism that fixes~$a$ and sends~$c$ to~$1$ and $\preccurlyeq$ 
is the partial order induced on $K_2^2$ by that of $\twiddle{\alg{MS}}$ and  
 \[
r = K_2^2 {\setminus} \{ 0c, 01, 10, c1\}. 
\]
Using the Test Spaces Method, we find a minimal algebra $\E(\X)$ whose underlying lattice is a $7$-element chain (see Fig.~\ref{fig:EXK2K3}). In this case, the $2$-generated $\class{K}_2$-free algebra has cardinality $414$ and is sufficiently small  for the algebraic approach to deliver the same solution. By contrast, computer calculations tell us that the $2$-generated $\class{K}_3$-free algebra has $3\,059$ elements and the algebraic approach is no longer feasible. We obtain, however, a strongly dualising alter ego for $\alg{K_3}$ 
\[
\twiddle{\alg{K_3}} =  (\{0,b,c,d,1\}; \{0,1\},\{0,d,1\}, h, \preccurlyeq , s),
\]
where $h$ is the only non-identity endomorphism of $\alg{K_3}$ which fixes~$d$, sends~$b$ to~$d$ and~$c$ to~$1$, the partial order $\preccurlyeq$ 
is the partial order induced on $K_3^2$ by that of $\twiddle{\alg{MS}}$, and  
\[
s = \{ 00, 0b, 0d, d0, db, dc, dd, d1, 1d, 11 \}.
\] 
The  Test Spaces Method then produces a minimal algebra $\E(\X)$ with $9$ elements (see Fig.~\ref{fig:EXK2K3}).  
\begin{figure} [ht]
\begin{center}
\begin{tikzpicture}[scale=.8]  
\node (1) at (0,0) {};
\node (2) at (0,.75){};
\node (3) at (0,1.5) {};
\node (4) at (0,2.25) {};
\node (5) at (0,3) {};
\node (6) at (0,3.75) {};
\node (7) at (0,4.5) {};

		\draw (1) circle (2pt);
		\draw (2) circle (2pt);
		\draw  (3) circle (2pt);
		\draw  (4)  circle (2pt);
	      \draw  (5) circle (2pt);
		\draw  (6) circle (2pt);
		\draw  (7) circle (2pt);

\draw [shorten <=-1pt, shorten >=-1pt] (1) -- (2);
\draw [shorten <=-1pt, shorten >=-1pt] (2) -- (3);
\draw [shorten <=-1pt, shorten >=-1pt] (3) -- (4);
\draw [shorten <=-1pt, shorten >=-1pt] (4) -- (5);
\draw [shorten <=-1pt, shorten >=-1pt] (5) -- (6);
\draw [shorten <=-1pt, shorten >=-1pt] (6) -- (7);

\draw [<->,dashed]  (1) .. controls +(-1,.6) and +(-1,-.6) .. (7);
\draw [<->,dashed]  (2) .. controls +(.8,1) and +(.8,-1) .. (6);
\draw [<-,dashed]  (2) .. controls +(-.6,1) and +(-.6,-1) .. (5);
\draw [<->,dashed]  (3) .. controls +(.4,.3) and +(.4,-.3) .. (4);
\end{tikzpicture}
\hspace*{1.5cm}
\begin{tikzpicture}

		\node (1) at (5,0) {}; 
	  	\node (3) at (5,2.25) {};  
		\node (5) at (5,.75)  {};  
		\node (7) at (4,1.5) {};  
		\node (9) at (6,1.5) {};  
		\node (10) at (3,2.25) {}; 
     		\node (12) at (4,3) {}; 
		\node (13) at (4,3.75) {}; 
		\node (14) at (4,4.5) {}; 

		\draw (1) circle (2pt);
		\draw  (3) circle (2pt);
	\draw  (5) circle (2pt);
		\draw  (7) circle (2pt);
		\draw  (9) circle (2pt);
       	\draw  (10) circle (2pt);
	\draw  (12) circle (2pt);
	\draw  (13) circle (2pt);
		\draw  (14) circle (2pt);
%
   		\draw [shorten <=-1pt, shorten >=-1pt] (1) -- (5);
		\draw [shorten <=-2pt, shorten >=-2pt] (5) -- (7);
		\draw [shorten <=-2pt, shorten >=-2pt] (5) -- (9);
		\draw [shorten <=-2pt, shorten >=-2pt] (7) -- (3);
		\draw [shorten <=-2pt, shorten >=-2pt] (9) -- (3);
		\draw [shorten <=-2pt, shorten >=-2pt] (3) -- (12);
		\draw [shorten <=-2pt, shorten >=-2pt] (7) -- (10);
		\draw [shorten <=-2pt, shorten >=-2pt] (12) -- (10);
		\draw [shorten <=-1pt, shorten >=-1pt] (12) -- (13);
		\draw [shorten <=-1pt, shorten >=-1pt] (14) -- (13);


\draw [<->,dashed]  (1) .. controls +(3,1) and +(3,-1) .. (14);
\draw [<->,dashed]  (13) .. controls +(-2.5,.5) and +(-2.5,-.5) .. (5);
\draw [->,dashed]  (3) -- (10);
\draw [->,dashed]  (7) -- (9);
\draw [->,dashed]  (12) .. controls +(.4,.1) and +(.4,.1) .. (3);
\draw [->,dashed]  (10)--(9); 
\end{tikzpicture}
\end{center}
\caption{ $\E(\X)$ for $ \K_2$ (left) and $\K_3$ (right)} \label{fig:EXK2K3}
\end{figure}
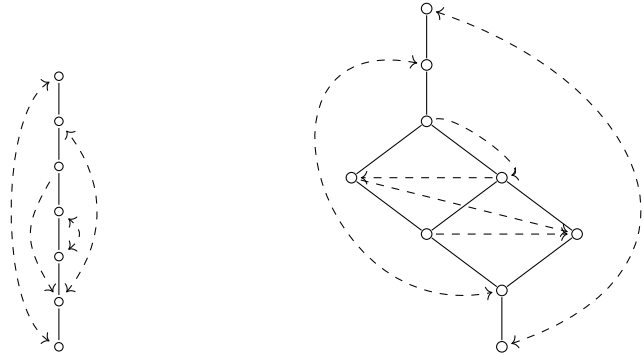
\end{Kiex}


\begin{DSex} \label{DSex} 
 
Background on this example and on the natural duality for double Stone algebras can be found in~\cite[Chapter 4, Theorem 3.14]{CD98} and~\cite{HaPr08}; the latter gives references to the original literature. 
 
Fig.~\ref{fig:dS} depicts the algebra 
$\alg{dS} =( \{0, a, b, 1\}; \wedge, \vee, {^*}, {^+}, 0, 1)$, 
generating the quasivariety of \defn{double Stone algebras} and its dualising alter ego
 $\twiddle{\alg{dS} }= (\{ 0,a,b,1\}; d,u,\preccurlyeq)$.
  Here $^*$ and $^+$  denote a pseudocomplement and a dual pseudocomplement, respectively.  
  On any  structure $\X \in \IScP(\twiddle{\alg{dS}})$, the maps $d$ and $u$ on~$X$ 
  are uniquely  determined by the partial order $\preccurlyeq$:  $d$ (respectively $u$) 
sends each element of $X$ to the unique  minimal point below it 
(respectively maximal point above it). 
 Accordingly we shall not show the action of these maps in our diagrams.

\begin{figure} [ht]
\begin{center}
	\begin{tikzpicture}
	
	  	\node [label=above:$1$] (1) at (-3,2) {};
		\node [label=right:$b$] (b) at (-3,1)  {};
  		\node [label=right:$a$] (a) at (-3,0) {};
		\node [label=below:$0$] (0) at (-3,-1)  {};

		\draw [shorten <=-1pt, shorten >=-1pt] (1) -- (b);
		\draw [shorten <=-1pt, shorten >=-1pt] (b) -- (a);
		\draw [shorten <=-1pt, shorten >=-1pt] (0) -- (a);

		\node at (-4,1.5) {$*$};
		\draw [<->]  (1) .. controls +(-1,0) and +(-1,0) .. (0);
		\draw [->]  (b) .. controls +(-.8,0) and +(-.8,0) .. (0);
		\draw [->]  (a) .. controls +(-.6,0) and +(-.6,0) .. (0);
		\node at (-2,1.5) {$+$};
		\draw [dashed,<->]  (1) .. controls +(1,0) and +(1,0) .. (0);
		\draw [dashed,->]  (a) .. controls +(.8,0) and +(.8,0) .. (1);
		\draw [dashed,->]  (b) .. controls +(.6,0) and +(.6,0) .. (1);

		\draw (1) circle [radius=2pt];
		\draw (b) circle [radius=2pt];
		\draw (a) circle [radius=2pt];
		\draw (0) circle [radius=2pt];

   		\node [label=below:$1$] (a1) at (2.8,-.2) {};
		\node [label=above:$b$] (ab) at (4,.8)  {};
  		\node [label=below:$a$] (aa) at (4,-.2) {};
		\node [label=below:$0$] (a0) at (1.6,-.2)  {};

		\draw [shorten <=-1pt, shorten >=-1pt] (aa) -- (ab);
		
		\node at (2.5,2) {$u$};
		\draw [->] (2.3,1.8) to (2.3,2.2);
		\node at (3.1,2) {$d$};
		\draw [dashed,->] (3.3,2.2) to (3.3,1.8);
		\draw [->]  (a0) .. controls +(-.5,-0.5) and +(-.5,0.5) .. (a0);
		\draw [->]  (a1) .. controls +(-.5,-0.5) and +(-.5,0.5) .. (a1);
		\draw [->]  (aa) .. controls +(-.5,0) and +(-.5,0) .. (ab);

		\draw [dashed, ->]  (a0) .. controls +(.5,0.5) and +(.5,-0.5) .. (a0);
		\draw [dashed,->]  (a1) .. controls +(.5,0.5) and +(.5,-0.5) .. (a1);
		\draw [dashed,->]  (ab) .. controls +(.5,0) and +(.5,0) .. (aa);
		\draw (a1) circle [radius=2pt];
		\draw (ab) circle [radius=2pt];
		\draw (aa) circle [radius=2pt];
		\draw (a0) circle [radius=2pt];
	\end{tikzpicture}
\end{center}
\caption{The generating algebra $\alg{dS} $ and its alter ego}\label{fig:dS}
\end{figure}
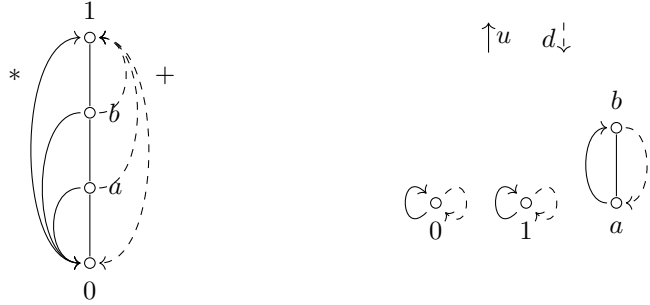
We now apply the Test Spaces Method, with $s=2$.

\begin{enumerate}

\item[1.]  \texttt{Compute $\D(\alg{dS})$.}

There are three endomorphisms of~$\alg{dS}$: the identity map, $\id$,  and  maps~$e_1$ sending  $a$ and $b$ to~$a$, and~$e_2$ sending  these elements to~$b$.  Equipping   $\{ \id, e_1,e_2\}$ with  the pointwise order from $(\twiddle{\alg{dS}};\preccurlyeq)$, we obtain a $3$-element chain with $e_1 \prec \id \prec e_2$.  The liftings of $d$ and~$u$ are the constant maps onto $e_1$ and $e_2$, respectively.

\begin{figure} [t]
\begin{center} 
\begin{tikzpicture}[scale=.85]
\node (lab) at  (7,-1)  {\small{$\twiddle{\alg{dS}}^2$ with $\X$ shaded}}; 
   		\node [label=below:{$00$}] (00) at (0,1) {}; 
		\node [label=below:{$0a$}] (0a) at (4,.5) {};  
  		\node [label=above:{$0b$}] (0b) at (4,1.5) {};  
		\node [label=below:{$01$}] (01) at (1,1) {};    
		\node [label=below:{$a0$}] (a0) at (5,.5) {};    
		\node [label=below:{$aa$}] (aa) at (10,0) {};   
  		\node [label=left:{$ab$}] (ab) at (9,1) {};    
		\node [label=below:{$a1$}] (a1) at (7,.5) {};  
		\node [label=above:{$b0$}] (b0) at (5,1.5) {};  
		\node [label=right:{$ba$}] (ba) at (11,1) {};   
  		\node [label=above:{$bb$}] (bb) at (10,2) {};  
		\node [label=above:{$b1$}] (b1) at (7,1.5) {};   
		\node [label=below:{$10$}] (10) at (2,1) {};  
		\node [label=below:{$1a$}] (1a) at (6,.5) {}; 
  		\node [label=above:{$1b$}] (1b) at (6,1.5) {};  
		\node [label=below:{$11$}] (11) at (3,1) {};  

\fill (00) circle [radius=2pt];
\draw (0a) circle [radius=2pt];
\draw (0b) circle [radius=2pt];
\draw (01) circle [radius=2pt];
\draw (a0) circle [radius=2pt];
\fill (aa) circle [radius=2pt];
\fill (ab) circle [radius=2pt];
\draw (a1) circle [radius=2pt];
\draw (b0) circle [radius=2pt];
\draw (ba) circle [radius=2pt];
\fill (bb) circle [radius=2pt];
\draw (b1) circle [radius=2pt];
\draw (10) circle [radius=2pt];
\draw (1a) circle [radius=2pt];
\draw (1b) circle [radius=2pt];
\draw (11) circle [radius=2pt];

	\draw [shorten <=-1pt, shorten >=-1pt]  (0a) -- (0b);
	\draw [shorten <=-1pt, shorten >=-1pt] (a0) -- (b0);
	\draw [shorten <=-1pt, shorten >=-1pt] (1a) -- (1b);
	\draw [shorten <=-1pt, shorten >=-1pt] (a1) -- (b1);
	\draw [shorten <=-2pt, shorten >=-2pt] (aa) -- (ab);
	\draw [shorten <=-2pt, shorten >=-2pt] (aa) -- (ba);
	\draw [shorten <=-2pt, shorten >=-2pt] (ab) -- (bb);
	\draw [shorten <=-2pt, shorten >=-2pt] (ba) -- (bb);

\draw[gray] (-0.5,-0.3) --  (-0.5, 2.3)-- (7.5,2.3) -- (7.5,-0.3) --  (-0.5,-0.3); 
\draw[gray] (8,0.7) --  (8, 1.4)-- (12,1.4) -- (12,0.7) -- (8,0.7); 
\draw[gray] (9.5,-.7) --  (9.5, 0.4)-- (10.5,0.4) -- (10.5,-0.7) -- (9.5,-.7) ; 
\draw[gray] (9.5,1.7) --  (9.5, 2.9)-- (10.5,2.9) -- (10.5,1.7) -- (9.5,1.7) ; 

	\end{tikzpicture}
\end{center}
\caption{Step  2 for double Stone algebras} \label{fig:dS2}
\end{figure}
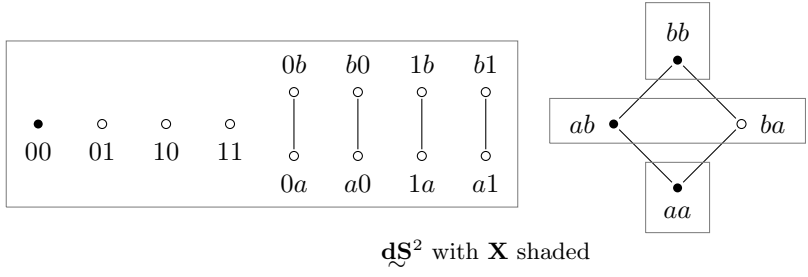
  
\item[2.]  \texttt{Determine a  TS-configuration $(\X,\gamma,\eta)$ with~$X$ of minimum size.}

\noindent 
The set 
$X=\{ 00, aa, ab,bb\} $
shown
in Fig.~\ref{fig:dS2} is 
the universe of a substructure of  $\twiddle{\alg{dS}}^2$. 
In fact
 it is a retract:
send each element of $X$ to itself, $ba$ to $ab$, and every other element to $00$.  Finally,
the $3$-element chain in $\X$ is isomorphic to $\D(\alg{dS})$.

\item[3.] \texttt{Determine the set $\mathcal{M}$ of maximal join-irreducible elements of $\mathcal{S}_{\X}$.}

The $\X$-substructure of $\X$ are determined by the sets $\{00\}$, $\{00, aa, bb\}$, and $X$. Hence $\X$ itself is $\X$-join-irreducible.

\item[5.] \texttt{Compute $\E(\X)$.}

We exploit the fact that  the  natural duality also operates as a restricted Priestley duality. 
The lattice reduct of $\E(\X)$ is isomorphic to the lattice of up-sets of $(X;\preccurlyeq)$, 
and hence to the product of a $2$-element chain and a $4$-element chain. 
The operations $^*$ and $^+$ of $\E(\X)$ are uniquely determined by its lattice order.  
Hence $\E(\X) \cong \alg{2} \times \alg{dS}$.

\end{enumerate}
\end{DSex}


\InvSex \label{InvSex}

The variety of \defn{involutive Stone algebras} is generated as a quasivariety by the following algebra~\cite{cg}:
\[
\alg{L_6} := (\{0, a, b, c, d, 1\}; \wedge, \vee, \sim, \nabla, 0, 1 ).
\]
The lattice order of $\alg{L_6} $ is shown in Fig.~\ref{fig:algL6}. 
The additional operations are defined as follows:  $\sim$ is a De Morgan negation  
which swaps $0$ and $1$ and swaps  $c$ and $d$  while fixing $a$ and $b$; 
$\nabla$ fixes $0$ and sends all other elements to $1$.

\begin{figure} [ht]
\begin{center}
	\begin{tikzpicture}
		\node [label=below:{$0$}] (0) at (.75,0) {};
		\node [label=left:{$a$}] (a) at (0,1.5) {};
		\node [label=right:{$b$}] (b) at (1.5,1.5) {};
		\node [label=left:{$c$}] (c) at (.75,2.25) {};
		\node [label=left:{$d$}] (d) at (.75,0.75) {};
		\node [label=above:{$1$}] (1) at (.75,3) {};

   		\draw [shorten <=-1pt, shorten >=-1pt] (0) -- (d);
		\draw [shorten <=-2pt, shorten >=-2pt] (d) -- (a);
   		\draw [shorten <=-2pt, shorten >=-2pt] (d) -- (b);
		\draw [shorten <=-2pt, shorten >=-2pt] (c) -- (a);
   		\draw [shorten <=-2pt, shorten >=-2pt] (c) -- (b);
		\draw [shorten <=-1pt, shorten >=-1pt] (c) -- (1);

		\draw (0) circle [radius=2pt];
		\draw (a) circle [radius=2pt];
		\draw (b) circle [radius=2pt];
		\draw (c) circle [radius=2pt];
		\draw (d) circle [radius=2pt];
		\draw (1) circle [radius=2pt];
		
		\draw [<->]  (0) .. controls +(-2,0.5) and +(-2,-0.5) .. (1);
		\draw [->]  (a) .. controls +(-.6,-0.55) and +(-.6,0.55) .. (a);
		\draw [->]  (b) .. controls +(.6,-0.55) and +(.6,0.55) .. (b);
		\draw [<->]  (c) -- (d);
		\draw [dashed,->]  (0) .. controls +(.6,-0.55) and +(.6,0.55) .. (0);
		\draw [dashed,->]  (d) .. controls +(2,0) and +(2,-0.5) .. (1);
		\draw [dashed,->]  (a) .. controls +(0,1) and +(-0.5,-0.5) .. (1);
		\draw [dashed,->]  (b) .. controls +(0,1) and +(0.5,-0.5) .. (1);
		\draw [dashed,->]  (c) .. controls +(0.3,0.2) and +(0.2,-0.2) .. (1);

	\end{tikzpicture}
\hspace*{1cm}
	\begin{tikzpicture}
		\node [label=above:{$0$}] (5) at (0,0.75) {};
		\node [label=below:{$a$}] (1) at (2,0) {};
		\node [label=above:{$b$}] (2) at (2,1.5) {};
		\node [label=left:{$c$}] (3) at (1.25,.75) {};
		\node [label=right:{$d$}] (4) at (2.75,.75) {};
		\node [label=above:{$1$}] (6) at (4,0.75) {};

   		\draw [shorten <=-2pt, shorten >=-2pt] (1) -- (3);
		\draw [shorten <=-2pt, shorten >=-2pt] (1) -- (4);
		\draw [shorten <=-2pt, shorten >=-2pt] (3) -- (2);
		\draw [shorten <=-2pt, shorten >=-2pt] (4) -- (2);

		\draw (1) circle [radius=2pt];
		\draw (2) circle [radius=2pt];
		\draw (3) circle [radius=2pt];
		\draw (4) circle [radius=2pt];
		\draw (5) circle [radius=2pt];
		\draw (6) circle [radius=2pt];
		\draw [->]  (5) .. controls +(-.5,-0.5) and +(-.5,0.5) .. (5);
		\draw [->]  (6) .. controls +(-.5,-0.5) and +(-.5,0.5) .. (6);
		\draw [->]  (2) .. controls +(-.5,0.7) and +(.5,0.7) .. (2);
		\draw [->]  (1) -- (2);
		\draw [->]  (3) .. controls +(0,.5) and +(-.5,0) ..  (2);
		\draw [->]  (4) .. controls +(0,.5) and +(.5,0) ..  (2);
		\draw [dashed,->]  (5) .. controls +(.5,-0.5) and +(.5,0.5) .. (5);
		\draw [dashed,->]  (6) .. controls +(.5,-0.5) and +(.5,0.5) .. (6);
		\draw [dashed,<->]  (1) .. controls +(-1.8,0) and +(-1.8,0) ..  (2);
		\draw [dashed,->]  (3) .. controls +(-.6,-0.55) and +(-.6,0.55) ..  (3);
		\draw [dashed,->]  (4) .. controls +(.6,-0.55) and +(.6,0.55) .. (4);
	\end{tikzpicture}

\caption{The algebra $\alg{L_6}$ and the partial order of its alter ego \label{fig:algL6}}
\end{center}
\end{figure}
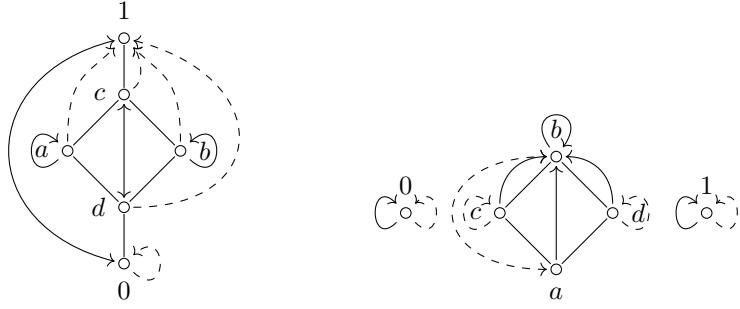

An alter ego of $\alg{L_6} $ can be found in~\cite[Example 4.15]{UnOpe}:
\[
\twiddle{\alg{L_6}} := (\{0, a, b, c, d, 1\}; g, h , \preccurlyeq),
\]
where
\[
g(x)= \begin{cases}
x & \mbox{if } x\in\{0,1\},\\
b & \mbox{otherwise} 
\end{cases}
\quad\mbox{and}\quad
h(x)= \begin{cases}
x & \mbox{if } x\in\{0,c,d,1\},\\
a & \mbox{if } x=b,\\
b & \mbox{if } x=a.
\end{cases}
\]
The partial order $\preccurlyeq$ is shown in Fig.~\ref{fig:algL6}. 

\begin{enumerate}

\item[1.]  \texttt{Compute $\D(\alg{L_6})$.} 
  
Writing $f \in \End (\alg{L_6})$ as the tuple $\big(f(0), f(a), f(b), f(c), f(d), f(1)\big)$, the four endomorphisms of $\alg{L_6}$ are $\id= (0, a, b, c, d, 1)$, $e_1 =(0, a, a, a, a, 1)$, $e_2 = (0, b, a, c, d, 1)$, and $e_3 =(0, b, b, b, b, 1)$.  The relation $\preccurlyeq$ lifts to the substructure $\D(\alg{L_6})$ to give  the partial order in which $\id$ and $e_2$ are incomparable,  $ e_1\prec \id \prec e_3$, and $e_1\prec e_2 \prec e_3$. 
 
\item[2.]  
 \texttt{Determine a  TS-configuration $(\X,\gamma,\eta)$ with~$X$ of minimum size.} 

Let $\X$ be the substructure of $\twiddle{\alg{L_6}}^2$ with universe $\{aa,cc,bb, ab,ba,00\}$, into which $\D(\alg{L_6}) $ embeds as $\X {\setminus} \{ cc,00\}$.

\item[3.] \texttt{Determine the set $\mathcal{M}$ of maximal join-irreducible elements of $\mathcal{S}_{\X}$.}

The $\X$-substructures of $\X$ form a chain and are determined by $\{00\}$,  $\{aa,cc,bb, 00\}$,  and~$X$. Hence ${\mathcal{M}= \{\X\}}$.

\item[5.] \texttt{Compute $\E(\X)$.}

A straightforward calculation shows that  $\E(\alg{X})$ is isomorphic to $\alg{2} \times \alg{10}$, where the $\class{DM}$-reduct of $\alg{10}$ is isomorphic to $\overline{\alg{D}_{42}}$ and $\nabla$ fixes~$0$ and sends all other elements to~$1$.

\end{enumerate}

Let us remark finally that the variety of \defn{Kleene--Stone algebras} can be identified with a subvariety of $\ISP(\alg{L_6})$: namely, $\ISP(\alg{L_5})$ where $\alg{L_5} = \alg{L_6} {\setminus} \{b\}$.  We shall omit the TSM analysis of this example, but have included data on cardinalities in Table~\ref{table:casestud} on page~\pageref{table:casestud}; the size of the free algebra on two generators  is taken from~\cite[Example 4.16]{UnOpe}.
  

\section{Concluding remarks}

In this paper we have addressed the computational problem of finding small algebras for checking admissibility in finitely generated quasivarieties. Using our Test Spaces Method we have been able to obtain smallest possible algebras for checking admissibility for several case studies that could not be handled by the algebraic method described in~\cite{MR12}. Even for some quasivarieties generated by very small algebras the algebraic approach  may fail, essentially because the required free algebras in these cases are too large. On the other hand, the Test Spaces Method capitalises on logarithmic features of the natural dualities for certain quasivarieties (see~\cite[Chapter~6]{CD98}), thereby obtaining significant computational benefits.   These benefits come  from two sources. Firstly, to implement the dual version of $\proc{SubPreHom}$ we do not need to work with the free algebras themselves, and it is therefore of no consequence if we are unable to compute them.  Secondly, the  $\proc{MinGenSet}$ procedure is more likely to be feasible to execute in its recast, dual, form. To emphasize the computational benefits of the approach, we have selected examples that are all 2-generated, have non-lattice operations with arity at most $1$, and differ more in terms of the size of their free algebras than in their algebraic structure.  Table~\ref{table:casestud} on page~\pageref{table:casestud}, summarising the data from our case studies, illustrates the benefits of the dual approach.  

A further remark deserves to be made concerning our examples in this paper. In no case was Step~4 in the Test Spaces Method required, since $\X$ always turned out to be join-irreducible in~$\mathcal{S}_\X$ for the TS-configuration $\X$ of minimum size obtained in Step~3. Examples where this is not the case can be easily constructed on the algebraic side, and we have therefore presented the theory to encompass the more general situation. Nevertheless, it would be of interest to find necessary and sufficient conditions on the generator $\M$ and/or its alter ego for the simpler situation to occur.

The applicability of the Test Spaces Method extends well beyond the examples considered in this paper.  In a further paper~\cite{CM18}  we employ the Test Spaces Method to study infinite chains of quasivarieties generated by finitely generated Sugihara chains: algebras for the logic $R$-mingle that have a binary non-lattice operation~(see~\cite{Dun70}). Quite sophisticated techniques from natural duality theory are used both to develop a general description of the natural dualities for these quasivarieties, and then to obtain a minimal algebra for checking admissibility.  In this setting we  deal with  quasivarieties whose generating algebras $\M$ are arbitrarily large, and where the minimal number of generators of~$\M$  also grows as $|M|$ increases.  In cases such as this, computational methods  are still helpful, but only in  the simplest cases.  Duality---in particular, pictorial representations of dual structures---has a role to play that extends beyond  computational considerations.

We have argued that the Test Spaces Method developed  is superior to the purely algebraic approach developed in~\cite{MR12}, both computationally and also, potentially, for more theoretical investigations.  But we should confront its limitations. It relies on the availability of a suitable strong  
duality for the considered finitely generated quasivariety. There do exist methods for automatically generating such a duality that are widely applicable (see, e.g.,~\cite[Sections 3.3 and 7.2]{CD98}); they encompass in particular all finitely generated lattice-based quasivarieties.  The resulting dualities may be complicated but, at least when the generating algebra is small, they may nonetheless allow the methods of this paper to be applied.


\end{document}